\documentclass[10pt]{amsart}

\usepackage[utf8]{inputenc}
\usepackage[T1]{fontenc}
\usepackage{lmodern}
\usepackage[a4paper,textwidth=15.2cm,textheight=23.7cm,centering]{geometry}
\usepackage{amsmath,amssymb,amsthm,mathtools}
\usepackage{graphicx}
\usepackage{mathrsfs}
\usepackage{enumitem}
\usepackage[colorlinks=true,citecolor=blue,linkcolor=blue,urlcolor=blue]{hyperref}
\usepackage[all]{nowidow}

\numberwithin{equation}{section}
\allowdisplaybreaks

\setlist{nosep,leftmargin=*}
\makeatletter
\@namedef{subjclassname@2020}{\textup{2020} Mathematics Subject Classification}

\makeatother

\makeatletter
\def\thm@space@setup{%
  \thm@preskip=5pt plus 1pt minus 1pt
  \thm@postskip=5pt plus 1pt minus 1pt}
\makeatother

\newtheorem{theorem}{Theorem}[section]
\newtheorem{proposition}[theorem]{Proposition}
\newtheorem{lemma}[theorem]{Lemma}
\newtheorem{corollary}[theorem]{Corollary}

\theoremstyle{definition}
\newtheorem{definition}[theorem]{Definition}
\newtheorem{remark}{Remark}[section]

\DeclareMathOperator{\tr}{tr}

\DeclareMathOperator{\spanop}{span}

\newcommand{\R}{\mathbb R}
\newcommand{\C}{\mathbb C}

\newcommand{\B}{\mathbb B}

\newcommand{\norm}[1]{\left\|#1\right\|}

\title[A Lewy theorem for harmonic quasiregular maps]{A Lewy theorem for harmonic quasiregular mappings in three-space}

\author{David Kalaj}
\address{David Kalaj, Faculty of Natural Sciences and Mathematics, University of Montenegro, Podgorica, Montenegro}
\email{davidk@ucg.ac.me}

\author{Jian-Feng Zhu}
\address{Jian-Feng Zhu, Department of Mathematics, Shantou University, Shantou, Guangdong 515063, P. R. China}
\email{flandy@stu.edu.cn}

\date{July 2026}



\begin{document}

\begin{abstract}
Lewy's classical theorem asserts that a one-to-one planar harmonic mapping has nonvanishing Jacobian. We prove a three-dimensional bounded-distortion analogue: if
\[
        f:\Omega\subset \mathbb R^3\to \mathbb R^3
\]
is nonconstant, sense-preserving, quasiregular, and harmonic componentwise, then \(J_f>0\) throughout \(\Omega\). Thus, harmonic quasiconformal mappings between domains in three-space are local harmonic diffeomorphisms.

The new point is the Lewy-type differential conclusion \(J_f\neq0\), not merely topological local invertibility, which is already known for sufficiently smooth quasiregular mappings. The proof is by blow-up. A hypothetical zero of \(J_f\) produces a nonconstant homogeneous harmonic polynomial quasiregular mapping \(P:\mathbb R^3\to\mathbb R^3\) of degree \(m>1\). We exclude such homogeneous blow-ups by a second-order trace identity for \(J_P|_{S^2}\). After normalizing the first jet at a positive minimum, the identity gives a negative spherical trace, contradicting the second derivative test at a minimum. We also prove that no homogeneous harmonic polynomial mapping \(\mathbb R^3\to\mathbb R^3\) of degree greater than one is injective, derive an affine Liouville theorem for entire harmonic quasiregular mappings in \(\mathbb R^3\), and obtain a local analytic Riemannian variant between oriented three-manifolds.
\end{abstract}

\maketitle

\section{Introduction}

\subsection{Lewy's theorem, higher-dimensional failure, and bounded distortion}

Lewy's theorem \cite{Lewy1936} asserts that a one-to-one complex-valued
harmonic mapping has nonvanishing Jacobian. In particular, every planar
harmonic homeomorphism is a local diffeomorphism. This result is one of the
classical starting points of the theory of planar harmonic mappings: a
harmonic homeomorphism in the plane cannot fold infinitesimally. In the planar
case this phenomenon is closely tied to the special structure of harmonic
functions as real parts of holomorphic functions, and to the fact that the
Jacobian can be written in terms of the analytic and anti-analytic parts.
It is therefore natural to ask which part of Lewy's theorem is topological,
which part is analytic, and which part is genuinely two-dimensional.

In dimensions at least three, the direct analogue of Lewy's theorem is false.
Wood \cite{Wood1991} constructed harmonic homeomorphisms with vanishing
Jacobian. Thus, harmonicity and global injectivity alone do not force local
invertibility, once the target and domain have dimension at least three. This
failure should be contrasted with Lewy's later theorem for harmonic gradients
in three dimensions \cite{Lewy1968}, and with the subsequent work of Gleason
and Wolff on harmonic gradient maps \cite{GleasonWolff}. These results show
that additional analytic structure may restore nondegeneracy, but they do not
cover general harmonic vector-valued mappings.

The purpose of this paper is to prove that, in dimension three, Lewy's
nondegeneracy conclusion is restored by a different structural assumption:
bounded distortion. We use the standard convention that a nonconstant mapping
\(f\in W^{1,n}_{\operatorname{loc}}(\Omega,\R^n)\) is sense-preserving
\(K\)-quasiregular if \(J_f\ge0\) a.e. in $\Omega$ and
\begin{equation}\label{distortion}
        \norm{Df(x)}^n\le KJ_f(x)
        \qquad\text{for a.e. }x\in\Omega .
\end{equation}
A quasiconformal mapping is a homeomorphic quasiregular mapping. Quasiregular
mappings are the higher-dimensional analogue of holomorphic functions with
bounded distortion. In particular, nonconstant quasiregular mappings are open
and discrete, and they enjoy a compactness theory due to Reshetnyak, see
\cite{Reshetnyak,Rickman,Vaisala}. In the present paper, the inequality
\eqref{distortion} plays a very concrete role: at a smooth point, if
\(J_f=0\), then \(Df=0\). Hence, a zero of the Jacobian is not merely a loss of
one singular value, it is a genuine critical point of the full differential.

We recall at this point a general smoothness theorem from quasiregular mapping
theory. A theorem of Martio--Rickman--V\"ais\"al\"a implies that every
nonconstant \(C^{n/(n-2)}\)-smooth quasiregular mapping in \(\R^n\),
\(n\ge3\), is locally homeomorphic. In particular, every nonconstant
\(C^3\)-smooth quasiregular mapping in \(\R^3\) has empty branch set. See
Rickman's monograph \cite[p. 12]{Rickman}, the original topological and metric
work \cite{MRV}, and the discussion in Bonk--Heinonen \cite{BonkHeinonen}. See
also Kaufman--Tyson--Wu \cite{KaufmanTysonWu} for a clear formulation of the
Martio--Rickman--V\"ais\"al\"a theorem and higher-dimensional sharpness results.
Since harmonic quasiregular mappings are real analytic, the topological
absence of branch points in the smooth three-dimensional setting is therefore
not the new part of our result. Our contribution is the stronger differential
Lewy conclusion: the Jacobian itself cannot vanish.

This distinction is essential. A smooth quasiregular local homeomorphism may
have critical points. For example, the radial stretch
\(x\mapsto |x|^2x\) is a polynomial quasiconformal homeomorphism of \(\R^n\),
\(n\ge3\), but its differential and Jacobian vanish at the origin. Thus, local
homeomorphism, even with smoothness and bounded distortion, does not imply
Jacobian nonvanishing. The harmonicity assumption in the theorem below rules
out precisely this type of homogeneous critical behavior.

\begin{theorem}
\label{thm-main}
Let \(\Omega\subset\R^3\) be a domain, and let
\(f:\Omega\to\R^3\) be a nonconstant sense-preserving quasiregular mapping
whose coordinate functions are harmonic. Then
\[
        J_f(x)>0,
        \qquad x\in\Omega .
\]
Consequently, \(f\) is locally a real-analytic diffeomorphism. In particular,
if \(f:\Omega\to\Omega'\) is a quasiconformal harmonic homeomorphism between
three-dimensional domains, then \(J_f\ne0\) everywhere.
\end{theorem}

Thus, Theorem~\ref{thm-main} should be viewed as a Jacobian nonvanishing
theorem, or a bounded-distortion form of Lewy's theorem. The branch-set
consequence is not the novelty in the smooth three-dimensional quasiregular
class. The novelty is the differential conclusion \(J_f>0\). No smallness condition is imposed on the distortion
constant. This is important, because many positive higher-dimensional Jacobian
estimates for harmonic quasiconformal mappings require additional quantitative
hypotheses, for example the bounds studied in \cite{BozinMateljevic}.
Here the conclusion is qualitative and holds for every finite quasiregular
distortion in dimension three. Related work on planar and spatial harmonic
quasiconformal mappings, boundary correspondence, and mappings with controlled
Laplacian includes
\cite{AstalaIwaniecMartin,AstalaManojlovic,Pavlovic2002,Kalaj2015,KalajSaksman2019}.
Martin's curved-metric form of Lewy's theorem \cite{Martin2016} is another
example in which additional geometric structure restores nondegeneracy.

The problem addressed here is therefore not a routine consequence of either
classical Lewy theory or general quasiregular topology. General smooth
quasiregular theory gives local homeomorphism, but it does not rule out
critical local homeomorphisms. Harmonic mapping theory gives analytic Taylor
expansions, but, by Wood's example, it does not by itself prevent the Jacobian
from vanishing. The result below identifies the missing mechanism: bounded
distortion transfers a zero of the Jacobian to a full critical point, and
harmonicity then forces the first tangent map to be a homogeneous harmonic
polynomial. The core of the paper is the observation that such a tangent map is
forbidden in dimension three by a second-order trace obstruction on the unit
sphere. This converts a Lewy-type question into a finite-dimensional spherical
first-jet problem and resolves it simultaneously for all homogeneous degrees.

Several nearby statements help explain the scope of the theorem. If one assumes
only that a harmonic map is a local diffeomorphism, then the Jacobian is
nonzero by definition, but this says nothing about whether local invertibility
or differential nondegeneracy can be forced from global or analytic hypotheses.
If one assumes global injectivity but no bounded distortion, Wood's construction
shows that the answer is negative. If one assumes quasiconformality, then the
map is a homeomorphism and has bounded distortion, and
Theorem~\ref{thm-main} gives the Lewy conclusion. The slightly stronger point
of the theorem is that the homeomorphism assumption is unnecessary for the
Jacobian conclusion: quasiregularity and harmonicity already force
\(J_f>0\) in the sense-preserving case.

This distinction is important for the structure of the proof. A quasiconformal
harmonic homeomorphism may a priori fail to be a local diffeomorphism if its
derivative degenerates. A smooth quasiregular mapping in \(\R^3\) is already
locally homeomorphic by the general theorem cited above, but it may still have
zero Jacobian. The theorem says that harmonicity prevents this differential
degeneracy. In this sense the result is a genuine Lewy-type strengthening of
smooth quasiregular local invertibility.

The theorem can also be compared with results for harmonic gradients. A map of
the form \(\nabla u\), where \(u\) is harmonic, has a symmetric derivative and
satisfies additional algebraic identities. Lewy's theorem for harmonic
gradients and the work of Gleason and Wolff exploit this special structure. In
our setting no symmetry of \(Df\) is assumed. The replacement is the
quasiregular distortion inequality and the spherical trace obstruction for the
leading homogeneous term. Thus, the result belongs simultaneously to the theory
of harmonic mappings and to quasiregular mapping theory, but the central new
calculation is neither a standard elliptic estimate, nor a standard topological
argument.

Finally, Theorem~\ref{thm-main} should be read as a local theorem. The affine
Liouville theorem for entire mappings is a consequence, not the source, of the
argument. The local statement is stronger, it applies in an arbitrary domain
and says that no critical point can occur, even before any global boundary or
properness assumption is imposed. This local character is what makes the
homogeneous blow-up method natural. It also makes the argument flexible under
changes of local coordinates. Since the leading part of the harmonic map
equation in normal coordinates is the Euclidean Laplacian, the same proof gives
an analytic Riemannian version for quasiregular harmonic maps between oriented
three-dimensional Riemannian manifolds, see
Corollary~\ref{cor-riemannian-lewy}.

\begin{remark}
The assumptions in Theorem~\ref{thm-main} have distinct functions. Harmonicity
makes the map real analytic and makes the first nonzero Taylor term a
homogeneous harmonic polynomial. Quasiregularity supplies the open-discrete
property and passes to locally uniform blow-up limits. The pointwise form of
bounded distortion forces \(Df(a)=0\) at a hypothetical zero of \(J_f\). The
orientation assumption is only a normalization, the corresponding statement for
orientation-reversing quasiconformal harmonic homeomorphisms follows by
reflecting the target.
\end{remark}

\subsection{The homogeneous obstruction}

The proof of Theorem~\ref{thm-main} is local and proceeds by blow-up. Suppose
that \(J_f(a)=0\). Since the coordinate functions of \(f\) are harmonic, they
are real analytic. Since \eqref{distortion} gives \(Df(a)=0\), the first
nonzero Taylor term of \(f-f(a)\) has degree \(m\ge2\) and
\[
        f(a+h)-f(a)=P_m(h)+O(|h|^{m+1}).
\]
The leading term \(P_m:\R^3\to\R^3\) is a nonconstant homogeneous harmonic
polynomial mapping. The rescalings
\[
        f_r(x)=\frac{f(a+rx)-f(a)}{r^m}
\]
converge locally uniformly to \(P_m\), and Reshetnyak compactness implies that
\(P_m\) is again quasiregular. Thus, a zero of the Jacobian of a harmonic
quasiregular map would produce a homogeneous harmonic quasiregular tangent
map. The main analytic work of the paper is to rule out precisely these
tangent maps in dimension three.

\begin{theorem}
\label{thm-homogeneous-qr-obstruction}
Let \(m>1\). There is no nonconstant homogeneous harmonic polynomial mapping
\[
        P:\R^3\to\R^3
\]
of degree \(m\) which is quasiregular. More precisely, there is no homogeneous
harmonic polynomial mapping of degree \(m>1\) whose Jacobian is strictly
positive on \(S^2\). The corresponding statement with strictly negative
Jacobian follows by reversing the orientation in the target.
\end{theorem}

Theorem~\ref{thm-homogeneous-qr-obstruction} is the core theorem of the
paper. For a homogeneous map \(P\) of degree \(m\), the Jacobian is homogeneous
of degree \(3m-3\). Hence, local quasiregularity away from the origin is
controlled by the behavior of \(J_P\) on the unit sphere. If a smooth
homogeneous quasiregular map had a zero of \(J_P\) on \(S^2\), the distortion
inequality would force \(DP=0\) at that point. Euler's identity would then
force \(P=0\) on the corresponding ray, contradicting discreteness. Thus, a
homogeneous quasiregular map would have \(J_P>0\) on \(S^2\), after possibly
reversing orientation. Theorem~\ref{thm-homogeneous-qr-obstruction} therefore
reduces to excluding positive homogeneous harmonic Jacobians on the sphere.

The new ingredient is a second-order trace identity for \(J_P\) on the sphere.
Assume, for contradiction, that \(J_P>0\) on \(S^2\), and take a positive
minimum. After rotating the domain and applying an orientation-preserving
linear change in the target, we may arrange
\[
        P(N)=e_3,
        \qquad
        P_x(N)=e_1,
        \qquad
        P_y(N)=e_2 .
\]
In these normalized coordinates, the trace identity gives
\[
        \Delta_{S^2}J_P(N)<0,
\]
whereas a minimum requires the trace of the Hessian to be nonnegative. The
contradiction is purely second-order. It does not depend on estimating the
quasiregular distortion constant, the constant \(K\) disappears after the
homogeneous reduction has forced the sign of \(J_P\).

\begin{remark}
A homogeneous harmonic polynomial map is finite-dimensional after restricting
to \(S^2\), so one might expect a proof by explicit spherical harmonics. The
trace identity gives a more invariant replacement for such a computation. It
uses only the first-jet normalization at a minimum, the spherical harmonic
eigenvalue equation, and the determinant expansion of the spherical Jacobian.
The sign that emerges is special, in the two-dimensional tangent plane the
relevant quadratic part becomes a sum of squares with the correct sign. This is
the mechanism that fails in higher dimensions.
\end{remark}

The normalization at the minimum deserves one further comment. If
\(J_P(N)>0\), then the three vectors \(P(N),P_x(N),P_y(N)\) are linearly
independent, with the orientation fixed by the sign of \(J_P(N)\). Applying a
linear map of positive determinant in the target sends the value and first
partial derivatives to the standard triple \(e_3,e_1,e_2\). Since the target
map has positive determinant, it preserves positivity of the Jacobian. The
calculation may therefore be made in a coordinate system, in which the zeroth
and the first order data are exactly those of the identity on the radial and
tangential directions. Every possible counterexample passes through this same
normal form.

Once the first jet is normalized, the only local freedom relevant to the trace
of \(J_P\) is contained in the second derivatives of the spherical harmonic
components. These second derivatives are not arbitrary, the spherical
Laplacian equation fixes their traces. Proposition~\ref{prop-trace-formula}
keeps track of exactly the remaining freedom. The first-derivative equations
at a minimum remove the divergence-type terms \(T_\gamma\), and the harmonic
trace constraints convert the two-dimensional quadratic expression into a
negative contribution. This is why the proof is short after the trace formula
has been established.

There is a useful geometric way to summarize the obstruction. A positive
Jacobian on \(S^2\) would make the value and the two tangential derivatives of
\(P|_{S^2}\) form a moving positively oriented frame in \(\R^3\). At a point,
where the determinant of this frame is smallest, the frame cannot expand to
second order in every tangent direction. Harmonicity forces the average second
variation of the frame determinant to be negative. This conflicts with the
minimum principle. The proof is this geometric picture written as a determinant
identity.

\subsection{The technical route}

We now describe the proof at a slightly finer level. Let \(P:\R^3\to\R^3\)
be homogeneous harmonic of degree \(m\). The restriction \(P|_{S^2}\) is a
triple of spherical harmonics of degree \(m\). By Euler's identity, the radial
column of \(DP\) at a point \(\theta\in S^2\) is \(mP(\theta)\). Consequently,
the Euclidean Jacobian \(J_P(\theta)\) can be written, up to the positive
factor \(m\), as the determinant of the three vectors
\[
        P(\theta),\quad \mathrm{d}P_\theta X,
        \quad \mathrm{d}P_\theta Y,
\]
where \(X,Y\) are an oriented orthonormal frame of \(T_\theta S^2\). Thus, the
Jacobian is a first-jet determinant on the sphere. A zero of this determinant
means that the value and first tangential derivatives fail to span the target.

The proof of the homogeneous obstruction begins by assuming that this
first-jet determinant is positive on the sphere. At a positive minimum we
normalize the value and the first tangential derivatives by an
orientation-preserving linear map in the target. This loses no information:
positivity of the determinant is preserved, and the sign of the second-order
trace at the minimum is invariant under the normalization. After this
normalization, the second derivatives of the spherical harmonics are the only
free local data relevant to the trace of \(J_P\).

The trace identity of Proposition~\ref{prop-trace-formula} is then applied at
the normalized point. In general tangent dimension \(d=n-1\), it has the form
\[
        \Delta j(0)
        =m\left(\sum_\gamma T_\gamma^2-Q(H)-(d-1)(m-1)(m+d)\right),
\]
where \(j=J_P\circ\theta\). At a minimum the first derivative equations force
\(T_\gamma=0\). In dimension three, where \(d=2\), the remaining quadratic
term becomes nonnegative after the spherical harmonic trace constraints are
used. Thus, the right side is strictly negative. This is the sign miracle in
the proof.

It is useful to note what the proof does not require. We do not classify
homogeneous harmonic maps. We do not choose an explicit basis of spherical
harmonics. We do not use degree theory beyond elementary covering arguments.
All local computations are made at a single normalized point. This locality is
what makes the argument robust enough to rule out all homogeneous degrees
\(m>1\) at once.

\begin{remark}
For cubic maps one can choose an explicit basis of the seven-dimensional space
of harmonic cubics in \(\R^3\), normalize the first jet at the north pole, and
obtain a finite-dimensional problem. Such a formulation is useful for examples
and numerical searches, but it hides the general mechanism. The proof below
uses only the eigenvalue of degree \(m\) spherical harmonics and the first-jet
normalization. The term \((m-1)(m+2)\) appearing in dimension three is strictly
positive for every \(m>1\), so the same contradiction excludes all higher
homogeneous degrees.
\end{remark}

\begin{remark}
The restriction of a homogeneous harmonic polynomial of degree \(m\) to the
sphere lies in the finite-dimensional space \(H_m(S^2)\) of spherical
harmonics. For a map \(P=(P^1,P^2,P^3)\), the condition \(J_P(\theta)\ne0\) is
equivalent to the statement that the first jets of \(P^1,P^2,P^3\) span
\(\R\oplus T_\theta^*S^2\). Thus, a hypothetical homogeneous quasiregular map
would give a three-dimensional subspace of \(H_m(S^2)\), whose first-jet
evaluation is an isomorphism at every point of the sphere. The trace identity
proves that no such subspace can arise when \(m>1\). This reformulation is
especially useful in comparing the three-dimensional theorem with the
higher-dimensional discussion in Section~\ref{sec-higher}.
\end{remark}

\subsection{A topological companion obstruction}

We also prove a purely topological homogeneous obstruction. Although this
result is not needed for the proof of Theorem~\ref{thm-main}, it clarifies the
relation between homogeneous harmonic maps, spherical direction maps, and the
Jacobian obstruction.

\begin{theorem}
\label{thm-homogeneous-noninjectivity}
Let \(m>1\). There is no one-to-one homogeneous harmonic polynomial mapping
\[
        P:\R^3\to\R^3
\]
of degree \(m\).
\end{theorem}

Equivalently, there is no injective homogeneous harmonic polynomial mapping
\(P:\R^3\to\R^3\) of degree greater than one.

For even \(m\), this is immediate from \(P(-x)=P(x)\). For odd \(m\), an
injective homogeneous map would induce a homeomorphism
\[
        F:S^2\to S^2,
        \qquad
        F(\theta)=\frac{P(\theta)}{|P(\theta)|}.
\]
If \(J_P\) vanished somewhere, then some nontrivial spherical harmonic
\(a\cdot P\) would have a critical zero. The local nodal structure at a
critical zero of a spherical harmonic on \(S^2\) is incompatible with the fact
that
\[
        \{a\cdot P=0\}=F^{-1}(S^2\cap a^\perp)
\]
would be a topological circle. Hence, \(J_P\ne0\) on \(S^2\), and the analytic
obstruction applies.

\begin{remark}
Theorems~\ref{thm-homogeneous-qr-obstruction} and
\ref{thm-homogeneous-noninjectivity} have different logical roles. The first is
an analytic obstruction to bounded distortion and is the ingredient used in the
blow-up proof of Theorem~\ref{thm-main}. The second says that even without a
distortion inequality, injectivity of a homogeneous harmonic polynomial map in
\(\R^3\) would force a nonvanishing Jacobian on the sphere, and hence would
again contradict the same trace obstruction. Thus, the spherical trace identity
controls both quasiregular tangent maps and homogeneous one-to-one maps, but
through different intermediate mechanisms.
\end{remark}

\subsection{Sharpness of the strict positivity obstruction}

The homogeneous obstruction in Theorem~\ref{thm-homogeneous-qr-obstruction}
is a strict positivity statement. It cannot be strengthened by replacing
``strictly positive'' with ``nonnegative and nontrivial.'' The following
explicit cubic shows that the boundary between quasiregularity and failure of
bounded distortion is sharp already in dimension three.

\begin{theorem}
\label{thm-icosahedral-borderline}
There exists a homogeneous harmonic cubic mapping
\(Q:\R^3\to\R^3\) such that
\[
        J_Q(x)\ge0,\qquad x\in\R^3,
\]
while \(J_Q\not\equiv0\). Moreover, the zero set of \(J_Q\) on \(S^2\) is
exactly the set of twelve vertices of a regular icosahedron. In particular,
\(Q\) is not quasiregular.
\end{theorem}

Thus, Proposition~\ref{prop-no-positive-homogeneous-jacobian-R3} below is optimal in
a precise sense: every homogeneous harmonic Jacobian of degree \(m>1\) must
fail to be strictly one-signed on \(S^2\), but a harmonic cubic may sit exactly
on the boundary with \(J\ge0\) and isolated zeros. The example also explains
why the proof of Theorem~\ref{thm-homogeneous-qr-obstruction} must use
quasiregular discreteness after the strict-minimum obstruction. The maximum
principle alone excludes positive minima, not nonnegative borderline zeros.

\subsection{Global consequence and dimensional sharpness}

The local theorem has a global consequence. Every entire nonconstant
sense-preserving harmonic quasiregular map \(\R^3\to\R^3\) is affine. Indeed,
Theorem~\ref{thm-main} gives \(J_f>0\), and hence removes the branch set. Zorich's theorem promotes the
map to a quasiconformal automorphism of \(\R^3\), and the standard growth
estimate for quasiconformal automorphisms implies polynomial growth. The
highest homogeneous harmonic term at infinity is then excluded by
Theorem~\ref{thm-homogeneous-qr-obstruction} unless it has degree one. This
Liouville theorem should be compared with the elementary existence of nonlinear
entire harmonic diffeomorphisms of \(\R^3\) that are not quasiregular. Bounded
distortion is again the rigidity assumption.

We also record the part of the argument that is genuinely \(n\)-dimensional.
The blow-up reduction, the spherical Jacobian formula, and the first-jet
formulation extend to arbitrary \(n\). What does not extend naively is the
sign of the second-order trace identity: for \(n\ge4\) the quadratic term
becomes indefinite. This failure is not merely formal. In even dimensions
\(n\ge4\) there are explicit homogeneous harmonic cubic models with
\(J\ge0\) on the sphere and with zeros on two latitude spheres. We strengthen
this observation by classifying all \(O(n-1)\)-equivariant harmonic cubic
models: none has strictly positive Jacobian on the sphere, while the even-dimensional
ones give exactly the natural one-signed borderline examples in that class.

The higher-dimensional borderline examples also explain why we do not state an
\(n\)-dimensional Lewy theorem as a consequence of the present method. In
higher dimensions a critical zero of a scalar spherical harmonic may have a
more complicated nodal hypersurface without contradicting the topology in the
same elementary way as on \(S^2\), and the second-order determinant trace has
indefinite algebraic part. Thus, both ingredients used in dimension three--the
maximum-principle sign and the nodal-curve obstruction--are tied to the low
dimensional geometry of the two-sphere.

For this reason, Section~\ref{sec-higher} has a precise but limited purpose.
It isolates the statements that are truly dimension-free and then gives a
concrete family showing that one cannot simply replace \(S^2\) by
\(S^{n-1}\) in the sign argument. The main theorem of the paper remains the
three-dimensional Jacobian nonvanishing theorem. The higher-dimensional material is
included to make the boundary of the method clear and to prevent the reader
from interpreting the proof as a disguised dimension-free argument.

\begin{remark}
The proof can be read as the following chain of implications:
\[
\begin{gathered}
        J_f(a)=0
        \Longrightarrow
        Df(a)=0
        \Longrightarrow
        P_m\text{ homogeneous harmonic and quasiregular},\\
        P_m\text{ quasiregular}
        \Longrightarrow
        J_{P_m}>0\text{ on }S^2
        \Longrightarrow
        \Delta_{S^2}J_{P_m}<0\text{ at a minimum}.
\end{gathered}
\]
which is impossible. The first implication is bounded distortion, the second is
analyticity of harmonic functions, the third is Reshetnyak compactness and
discreteness, and the last implication is the new second-order trace identity.
This separation of roles is useful: all steps except the final sign computation
are stable in arbitrary dimension, while the final sign computation is
intrinsically three-dimensional.
\end{remark}

The paper is organized as follows. Section~\ref{sec-preliminaries} collects
quasiregular compactness, harmonic blow-ups, and the spherical reduction for
homogeneous maps. Section~\ref{sec-trace} proves the second-order trace
identity, with the full determinant calculation included in the proof so that
the normalization and sign conventions are completely transparent.
Section~\ref{sec-obstruction} proves the homogeneous obstructions in
\(\R^3\). Section~\ref{sec-lewy} proves Theorem~\ref{thm-main}, its analytic
Riemannian variants, and the global Liouville corollaries.
Section~\ref{sec-higher} proves the icosahedral
borderline theorem, gives the higher-dimensional first-jet formulation, and
classifies the equivariant harmonic cubic models.

\section{Blow-ups and spherical reductions}
\label{sec-preliminaries}

\subsection{Quasiregular mappings}

\begin{definition}
Let \(\Omega\subset\R^n\) be a domain. A nonconstant mapping
\(f:\Omega\to\R^n\) is called sense-preserving \(K\)-quasiregular if
\(f\in W^{1,n}_{\operatorname{loc}}(\Omega,\R^n)\), \(J_f\ge0\) a.e., and
\[
        \norm{Df(x)}^n\le KJ_f(x)
        \qquad\text{for a.e. }x\in\Omega .
\]
A quasiconformal mapping is a homeomorphic quasiregular mapping.
\end{definition}

We shall use two standard facts from Reshetnyak theory, see
\cite{Reshetnyak,Rickman,Vaisala}.

\begin{lemma}
\label{lem-reshetnyak}
The following hold.
\begin{enumerate}[label=(\roman*)]
\item A nonconstant quasiregular mapping is open and discrete.
\item If \(f_j:\Omega\to\R^n\) are \(K\)-quasiregular and converge locally
uniformly to a nonconstant mapping \(f\), then \(f\) is \(K\)-quasiregular.
\end{enumerate}
\end{lemma}

For smooth mappings, the distortion inequality can be evaluated at every
point.

\begin{lemma}
\label{lem-pointwise-distortion}
Let \(f:\Omega\to\R^n\) be \(C^1\), and suppose that
\[
        \norm{Df}^n\le KJ_f
\]
holds almost everywhere in \(\Omega\). Then the same inequality holds
everywhere. In particular, if \(J_f(a)=0\), then \(Df(a)=0\).
\end{lemma}

\begin{proof}
The function \(KJ_f-\norm{Df}^n\) is continuous. If it were negative at some
point, it would be negative on a small ball, contradicting the
almost-everywhere inequality. Hence, it is nonnegative everywhere. At a point
where \(J_f=0\), the inequality gives \(\norm{Df}=0\).
\end{proof}

\subsection{Homogeneous maps and spherical coordinates}

Let \(P:\R^n\to\R^n\) be homogeneous of degree \(m\). Euler's identity gives
\begin{equation}\label{Euler}
        DP(x)x=mP(x).
\end{equation}
Consequently,
\begin{equation}\label{homogeneous-jacobian}
        J_P(r\theta)=r^{n(m-1)}J_P(\theta),
        \qquad r>0,\quad \theta\in S^{n-1}.
\end{equation}
Thus, by \eqref{homogeneous-jacobian}, the distortion of \(P\) away from the origin is governed by its
restriction to the sphere.

\begin{lemma}
\label{lem-spherical-jacobian}
Let \(P:\R^n\to\R^n\) be homogeneous of degree \(m\), and let
\(X_1,\ldots,X_{n-1}\) be an oriented orthonormal basis of \(T_\theta S^{n-1}\),
so that \((X_1,\ldots,X_{n-1},\theta)\) is positively oriented in \(\R^n\).
Then
\begin{equation}\label{spherical-jacobian}
        J_P(\theta)
        =
        m\det\bigl(\mathrm{d}P_\theta X_1,\ldots,\mathrm{d}P_\theta X_{n-1},P(\theta)\bigr).
\end{equation}
In particular, \(J_P(\theta)\ne0\) if and only if the vectors
\[
        P(\theta),\ \mathrm{d}P_\theta X_1,\ldots,\mathrm{d}P_\theta X_{n-1}
\]
span \(\R^n\).
\end{lemma}

\begin{proof}
The columns of \(DP(\theta)\) in the oriented basis
\((X_1,\ldots,X_{n-1},\theta)\) are
\[
        \mathrm{d}P_\theta X_1,\ldots,\mathrm{d}P_\theta X_{n-1},DP(\theta)\theta.
\]
By \eqref{Euler}, the last column is \(mP(\theta)\). Taking determinants
gives \eqref{spherical-jacobian}.
\end{proof}

We shall use graph coordinates near the north pole. Put \(d=n-1\), write
\(u=(u_1,\ldots,u_d)\in\R^d\), and set
\begin{equation}\label{graph-coordinates}
        \theta(u)=(u,w(u)),
        \qquad
        w(u)=\sqrt{1-|u|^2}.
\end{equation}
If \(\phi\) is a smooth function on the sphere and we write again
\(\phi(u)=\phi(\theta(u))\), then
\begin{equation}\label{sphere-laplacian}
        \Delta_{S^d}\phi
        =
        \sum_{\alpha,\beta=1}^{d}(\delta_{\alpha\beta}-u_\alpha u_\beta)
        \partial_{\alpha\beta}\phi
        -
        d\sum_{\alpha=1}^{d}u_\alpha\partial_\alpha\phi .
\end{equation}
If \(\phi\) is the restriction to \(S^d\) of a homogeneous harmonic
polynomial of degree \(m\), then
\begin{equation}\label{spherical-eigenvalue}
        \Delta_{S^d}\phi=-\lambda\phi,
        \qquad
        \lambda=m(m+d-1).
\end{equation}

\subsection{The blow-up of a harmonic quasiregular mapping}

\begin{lemma}
\label{lem-analytic-blow-up}
Let \(f:\Omega\subset\R^n\to\R^n\) be harmonic and nonconstant, and let
\(a\in\Omega\). If \(Df(a)=0\), then
\[
        f(a+h)-f(a)=P_m(h)+O(|h|^{m+1})
\]
for some integer \(m\ge2\), where \(P_m:\R^n\to\R^n\) is a nonzero
homogeneous harmonic polynomial mapping of degree \(m\). Moreover,
\[
        f_r(x):=\frac{f(a+rx)-f(a)}{r^m}
\]
converges to \(P_m\) in \(C^\infty_{\operatorname{loc}}\) as \(r\downarrow0\).
\end{lemma}

\begin{proof}
Harmonic functions are real analytic. Let \(m\) be the least degree for which
at least one coordinate of the Taylor expansion of \(f-f(a)\) has a nonzero
homogeneous term. Since \(Df(a)=0\), one has \(m\ge2\). Applying the
Euclidean Laplacian to the Taylor series shows that the degree-\(m\) term is
harmonic componentwise. Rescaling the Taylor expansion gives locally uniform
convergence of \(f_r\) to \(P_m\), and elliptic regularity gives convergence
of all derivatives on compact subsets.
\end{proof}

\begin{proposition}
Let \(\Omega\subset\R^n\) be a domain and let \(f:\Omega\to\R^n\) be a
sense-preserving \(K\)-quasiregular harmonic mapping. If \(J_f(a)=0\) at
some point \(a\in\Omega\), then the first nonzero Taylor term of \(f-f(a)\)
at \(a\) is a nonconstant homogeneous harmonic polynomial mapping
\[
        P_m:\R^n\to\R^n,\qquad m\ge2,
\]
and \(P_m\) is \(K\)-quasiregular.
\end{proposition}

\begin{proof}
By Lemma~\ref{lem-pointwise-distortion}, \(Df(a)=0\). Lemma
\ref{lem-analytic-blow-up} produces the first nonzero homogeneous harmonic
term \(P_m\) and the rescalings
\[
        f_r(x)=\frac{f(a+rx)-f(a)}{r^m}.
\]
Domain dilations, target translations, and positive target homotheties
preserve the quasiregular distortion inequality, so each \(f_r\) is
\(K\)-quasiregular on its rescaled domain. Since \(f_r\to P_m\) locally
uniformly and \(P_m\) is nonconstant, Lemma~\ref{lem-reshetnyak} implies
that \(P_m\) is \(K\)-quasiregular.
\end{proof}

\section{The second-order trace identity}
\label{sec-trace}

The following identity is the computational heart of the proof. It is stated
in arbitrary dimension because it also explains why the argument closes in
\(\R^3\) but does not immediately close in higher dimensions. We shall use
the graph coordinates \eqref{graph-coordinates}, the local formula
\eqref{sphere-laplacian}, and the spherical eigenvalue relation
\eqref{spherical-eigenvalue}.

\begin{proposition}
\label{prop-trace-formula}
Let \(d=n-1\), and let \(P:\R^n\to\R^n\) be homogeneous harmonic of degree
\(m\). Use the local coordinates
\[
        \theta(u)=\bigl(u,\sqrt{1-|u|^2}\bigr)
\]
near \(N=e_n\). Suppose that, after an orientation-preserving linear change
in the target,
\begin{equation}\label{normalization}
        F(0)=e_n,
        \qquad
        F_\alpha(0)=e_\alpha,\qquad \alpha=1,\ldots,d,
\end{equation}
where \(F(u)=P(\theta(u))\). Put
\[
        j(u)=J_P(\theta(u)),
        \qquad
        H^i_{\alpha\beta}=\partial_{\alpha\beta}F^i(0),
\]
\[
        T_\gamma=\sum_{p=1}^{d}H^p_{p\gamma},
        \qquad
        Q(H)=\sum_{\gamma=1}^{d}\sum_{p,q=1}^{d}
              H^p_{q\gamma}H^q_{p\gamma},
        \qquad
        \mu=(m-1)(m+d).
\]
Then
\begin{equation}\label{first-derivative-j}
        \partial_\gamma j(0)=mT_\gamma,
\end{equation}
and
\begin{equation}\label{trace-formula}
        \Delta j(0)
        =
        m\left(
          \sum_{\gamma=1}^{d}T_\gamma^2
          -
          Q(H)
          -
          (d-1)\mu
        \right).
\end{equation}
\end{proposition}

\begin{proof}
We write
\[
        w(u)=\sqrt{1-|u|^2},
        \qquad
        \theta(u)=(u,w(u)).
\]
Thus, \(\theta(0)=e_n\). We also write
\[
        F_\alpha=\partial_\alpha F,
        \qquad
        \theta_\alpha=\partial_\alpha\theta.
\]
Define the \(n\times n\) matrix
\[
        C(u)=\bigl(F_1(u),\ldots,F_d(u),F(u)\bigr),
        \qquad
        s(u)=\det C(u).
\]
The point of this notation is that the Jacobian of \(P\) can be computed by
comparing the frame
\[
        \theta_1,\ldots,\theta_d,\theta
\]
in the domain with its image under \(DP\).

First, we compute the determinant of the domain frame. Since
\[
        \theta_\alpha=(e_\alpha,w_\alpha),
\]
the matrix with columns \(\theta_1,\ldots,\theta_d,\theta\) is
\[
        \begin{pmatrix}
        I_d & u\\
        \nabla w^T & w
        \end{pmatrix}.
\]
Therefore,
\[
        \det(\theta_1,\ldots,\theta_d,\theta)
        =
        w-\nabla w\cdot u.
\]
Because
\[
        \nabla w=-\frac{u}{w},
\]
we get
\[
        w-\nabla w\cdot u
        =
        w+\frac{|u|^2}{w}
        =
        \frac{w^2+|u|^2}{w}
        =
        \frac1w.
\]
Hence,
\[
        \det(\theta_1,\ldots,\theta_d,\theta)=\frac1{w(u)}.
\]
On the other hand,
\[
        DP(\theta(u))\theta_\alpha(u)=F_\alpha(u).
\]
Also, by Euler's identity for a homogeneous map of degree \(m\),
\[
        DP(\theta(u))\theta(u)=mP(\theta(u))=mF(u).
\]
Taking determinants, and using the orientation fixed by the frame above, we
obtain
\[
        J_P(\theta(u))\det(\theta_1,\ldots,\theta_d,\theta)
        =
        \det(F_1,\ldots,F_d,mF).
\]
Thus,
\[
        j(u)\frac1{w(u)}=ms(u),
\]
or equivalently
\begin{equation}\label{j-mws}
        j(u)=mw(u)s(u).
\end{equation}

At \(u=0\), the normalization \eqref{normalization} gives
\[
        F(0)=e_n,\qquad F_\alpha(0)=e_\alpha,\qquad \alpha=1,\ldots,d.
\]
Therefore,
\[
        C(0)=I,
        \qquad
        s(0)=1.
\]
Moreover,
\[
        w(0)=1,
        \qquad
        \nabla w(0)=0,
        \qquad
        \Delta w(0)=-d.
\]
Indeed,
\[
        w(u)=1-\frac12|u|^2+O(|u|^4).
\]
Applying the Laplacian to \eqref{j-mws} and using the product rule, we get
\[
        \frac1m\Delta j(0)
        =
        \Delta(ws)(0)
        =
        w(0)\Delta s(0)+s(0)\Delta w(0)+2\nabla w(0)\cdot\nabla s(0).
\]
Hence,
\begin{equation}\label{Deltaj-Deltas}
        \frac1m\Delta j(0)=\Delta s(0)-d.
\end{equation}

It remains to compute \(\Delta s(0)\). Since \(C(0)=I\), the standard first
and second differential formulas for the determinant at the identity give
\[
        s_\gamma(0)=\tr C_\gamma(0),
\]
and
\begin{equation}\label{det-second}
        s_{\gamma\gamma}(0)
        =
        \tr C_{\gamma\gamma}(0)
        +
        \bigl(\tr C_\gamma(0)\bigr)^2
        -
        \tr\bigl(C_\gamma(0)^2\bigr).
\end{equation}
Indeed,
\[
        D(\det)_I(A)=\tr A,
        \qquad
        D^2(\det)_I(A,A)=(\tr A)^2-\tr(A^2).
\]

We now compute the first trace. Since
\[
        C(u)=\bigl(F_1(u),\ldots,F_d(u),F(u)\bigr),
\]
we have
\[
        C_\gamma(u)=\bigl(F_{1\gamma}(u),\ldots,F_{d\gamma}(u),F_\gamma(u)\bigr).
\]
At \(u=0\), the last column is
\[
        F_\gamma(0)=e_\gamma.
\]
Since \(\gamma\le d\), this last column contributes zero to the last diagonal
entry. Therefore,
\begin{equation}\label{trace-Cgamma}
        \tr C_\gamma(0)
        =
        \sum_{p=1}^{d}F^p_{p\gamma}(0)
        =
        \sum_{p=1}^{d}H^p_{p\gamma}
        =
        T_\gamma.
\end{equation}
Together with \(\nabla w(0)=0\), this gives
\[
        \partial_\gamma j(0)=ms_\gamma(0)=mT_\gamma,
\]
which proves \eqref{first-derivative-j}.

We next compute
\[
        \Delta s(0)=\sum_{\gamma=1}^{d}s_{\gamma\gamma}(0).
\]
From \eqref{det-second},
\[
        \Delta s(0)
        =
        \sum_{\gamma=1}^{d}\tr C_{\gamma\gamma}(0)
        +
        \sum_{\gamma=1}^{d}T_\gamma^2
        -
        \sum_{\gamma=1}^{d}\tr\bigl(C_\gamma(0)^2\bigr).
\]
Thus, we have to evaluate the first and third terms.

Since each coordinate \(F^i\) is the restriction to the sphere of a
homogeneous harmonic polynomial of degree \(m\), it is a spherical harmonic
of degree \(m\). Hence,
\[
        \Delta_{S^d}F^i=-\lambda F^i,
        \qquad
        \lambda=m(m+d-1).
\]
At \(u=0\), the local formula for \(\Delta_{S^d}\) reduces to the Euclidean
tangent Laplacian, and therefore
\begin{equation}\label{second-spherical}
        \sum_{\gamma=1}^{d}F^i_{\gamma\gamma}(0)
        =
        -\lambda F^i(0).
\end{equation}
Differentiating the local expression
\[
        \Delta_{S^d}\phi
        =
        \sum_{\alpha,\beta=1}^{d}
        (\delta_{\alpha\beta}-u_\alpha u_\beta)\phi_{\alpha\beta}
        -
        d\sum_{\alpha=1}^{d}u_\alpha\phi_\alpha
\]
and evaluating at \(u=0\), we obtain
\begin{equation}\label{third-spherical}
        \sum_{\gamma=1}^{d}F^i_{\gamma\gamma\alpha}(0)
        =
        -(\lambda-d)F^i_\alpha(0).
\end{equation}
Indeed, the derivative of the coefficient
\((\delta_{\alpha\beta}-u_\alpha u_\beta)\) vanishes at \(u=0\), while
differentiating the first-order term gives \(-dF^i_\alpha(0)\).

Now
\[
        C_{\gamma\gamma}(u)
        =
        \bigl(F_{1\gamma\gamma}(u),\ldots,F_{d\gamma\gamma}(u),
              F_{\gamma\gamma}(u)\bigr).
\]
Therefore,
\[
        \tr C_{\gamma\gamma}(0)
        =
        \sum_{p=1}^{d}F^p_{p\gamma\gamma}(0)
        +
        F^n_{\gamma\gamma}(0).
\]
Summing over \(\gamma\), and using \eqref{third-spherical} with \(i=p\) and
\(\alpha=p\), we get
\[
        \sum_{\gamma=1}^{d}F^p_{p\gamma\gamma}(0)
        =
        \sum_{\gamma=1}^{d}F^p_{\gamma\gamma p}(0)
        =
        -(\lambda-d)F^p_p(0).
\]
By the normalization \(F_p(0)=e_p\), we have \(F^p_p(0)=1\). Thus,
\[
        \sum_{\gamma=1}^{d}F^p_{p\gamma\gamma}(0)=-(\lambda-d).
\]
Summing this over \(p=1,\ldots,d\), and using \eqref{second-spherical} with
\(i=n\), we obtain
\begin{equation}\label{trace-Cgammagamma}
        \sum_{\gamma=1}^{d}\tr C_{\gamma\gamma}(0)
        =
        -d(\lambda-d)-\lambda.
\end{equation}

It remains to compute the square term. For fixed \(\gamma\), the matrix
\(C_\gamma(0)\) has the block form
\[
        C_\gamma(0)
        =
        \begin{pmatrix}
        A_\gamma & e_\gamma\\
        b_\gamma^T & 0
        \end{pmatrix},
\]
where
\[
        (A_\gamma)^p_q=H^p_{q\gamma},
        \qquad
        (b_\gamma)_q=H^n_{q\gamma},
        \qquad
        p,q=1,\ldots,d.
\]
For such a block matrix,
\[
        \tr
        \begin{pmatrix}
        A & v\\
        b^T & 0
        \end{pmatrix}^{\!2}
        =
        \tr(A^2)+2b^Tv.
\]
Here \(v=e_\gamma\), so
\[
        b_\gamma^Te_\gamma=H^n_{\gamma\gamma}.
\]
Also,
\[
        \tr(A_\gamma^2)=\sum_{p,q=1}^{d}H^p_{q\gamma}H^q_{p\gamma}.
\]
Consequently,
\[
        \tr\bigl(C_\gamma(0)^2\bigr)
        =
        \sum_{p,q=1}^{d}H^p_{q\gamma}H^q_{p\gamma}
        +
        2H^n_{\gamma\gamma}.
\]
Summing over \(\gamma\), we obtain
\[
        \sum_{\gamma=1}^{d}\tr\bigl(C_\gamma(0)^2\bigr)
        =
        Q(H)+2\sum_{\gamma=1}^{d}H^n_{\gamma\gamma}.
\]
But \(H^n_{\gamma\gamma}=F^n_{\gamma\gamma}(0)\), and by
\eqref{second-spherical} with \(i=n\),
\[
        \sum_{\gamma=1}^{d}F^n_{\gamma\gamma}(0)
        =
        -\lambda F^n(0)
        =
        -\lambda,
\]
because \(F(0)=e_n\). Therefore,
\begin{equation}\label{trace-Cgamma-square}
        \sum_{\gamma=1}^{d}\tr\bigl(C_\gamma(0)^2\bigr)
        =
        Q(H)-2\lambda.
\end{equation}

Combining \eqref{det-second}, \eqref{trace-Cgamma},
\eqref{trace-Cgammagamma}, and \eqref{trace-Cgamma-square}, we get
\[
\begin{aligned}
        \Delta s(0)
        &=
        \sum_{\gamma=1}^{d}T_\gamma^2
        +
        \bigl[-d(\lambda-d)-\lambda\bigr]
        -
        \bigl[Q(H)-2\lambda\bigr]  \\
        &=
        \sum_{\gamma=1}^{d}T_\gamma^2
        -
        Q(H)
        -
        (d-1)\lambda
        +
        d^2 .
\end{aligned}
\]
Finally, by \eqref{Deltaj-Deltas},
\[
        \frac1m\Delta j(0)
        =
        \sum_{\gamma=1}^{d}T_\gamma^2
        -
        Q(H)
        -
        (d-1)\lambda
        +
        d^2-d.
\]
Since
\[
        \lambda-d=m(m+d-1)-d=(m-1)(m+d)=\mu,
\]
we obtain
\[
        \Delta j(0)
        =
        m\left(
          \sum_{\gamma=1}^{d}T_\gamma^2
          -
          Q(H)
          -
          (d-1)\mu
        \right),
\]
which is \eqref{trace-formula}.
\end{proof}

\section{Homogeneous obstructions in dimension three}
\label{sec-obstruction}

We now prove Theorem~\ref{thm-homogeneous-qr-obstruction}. The proof uses
only the second-order trace formula and the topology of quasiregular
mappings.

\begin{proposition}
\label{prop-no-positive-homogeneous-jacobian-R3}
Let \(m>1\). There is no homogeneous harmonic polynomial mapping
\(P:\R^3\to\R^3\) of degree \(m\) such that
\[
        J_P(\theta)>0,
        \qquad
        \theta\in S^2.
\]
The same conclusion with \(J_P<0\) follows after reversing the orientation in
the target.
\end{proposition}

\begin{proof}
Assume, to the contrary, that \(J_P>0\) on \(S^2\). Since \(S^2\) is compact,
\(J_P\) has a positive minimum. Rotate the domain so that the minimum point
is \(N=e_3\). By Lemma~\ref{lem-spherical-jacobian}, the three vectors
\[
        P(N),\quad P_x(N),\quad P_y(N)
\]
are linearly independent. Composing in the target with a linear map of
positive determinant, we may assume
\[
        P(N)=e_3,
        \qquad
        P_x(N)=e_1,
        \qquad
        P_y(N)=e_2.
\]
Use local coordinates
\[
        \theta(u,v)=\bigl(u,v,\sqrt{1-u^2-v^2}\bigr),
\]
and put \(F=P\circ\theta\), \(j=J_P\circ\theta\). Write
\[
        a=F^1_{uu}(0),
        \qquad
        b=F^1_{uv}(0),
        \qquad
        c=F^2_{uu}(0),
        \qquad
        d_0=F^2_{uv}(0).
\]
Because \(F^1\) and \(F^2\) are spherical harmonics and vanish at \(0\),
their second derivatives satisfy
\[
        F^1_{vv}(0)=-a,
        \qquad
        F^2_{vv}(0)=-c.
\]
In the notation of Proposition~\ref{prop-trace-formula}, with \(d=2\),
\[
        T_1=a+d_0,
        \qquad
        T_2=b-c.
\]
Since \((0,0)\) is a local minimum of \(j\), \eqref{first-derivative-j} gives
\[
        T_1=T_2=0.
\]
Hence,
\[
        d_0=-a,
        \qquad
        c=b.
\]
A direct substitution into \(Q(H)\) gives
\[
        Q(H)=4(a^2+b^2).
\]
Since \(d=2\), \(\mu=(m-1)(m+2)>0\). By \eqref{trace-formula},
\[
        \Delta j(0)
        =
        -m\left\{4(a^2+b^2)+(m-1)(m+2)\right\}<0.
\]
This contradicts the necessary condition \(\Delta j(0)\ge0\) at a local
minimum. Therefore, no such \(P\) exists. If \(J_P<0\), compose \(P\) with an
orientation-reversing linear isometry of the target to reduce to the positive
case.
\end{proof}

\begin{proof}[Proof of Theorem~\ref{thm-homogeneous-qr-obstruction}]
Suppose, to the contrary, that a nonconstant homogeneous harmonic polynomial
mapping \(P:\R^3\to\R^3\) of degree \(m>1\) is quasiregular. After a
reflection in the target, if necessary, assume that \(P\) is
sense-preserving. Since \(P\) is smooth, the distortion inequality holds
everywhere by Lemma~\ref{lem-pointwise-distortion}. If \(J_P(\theta)=0\) at
some point \(\theta\in S^2\), then \(DP(\theta)=0\). Euler's identity gives
\[
        0=DP(\theta)\theta=mP(\theta),
\]
and hence \(P(\theta)=0\). By homogeneity, \(P(r\theta)=0\) for every
\(r>0\), so the preimage of \(0\) contains a whole ray. This contradicts the
discreteness of nonconstant quasiregular mappings. Thus, \(J_P\) does not
vanish on \(S^2\). Since \(P\) is sense-preserving quasiregular, \(J_P\ge0\)
almost everywhere, and the continuity of \(J_P\) gives \(J_P\ge0\) everywhere.
Therefore, \(J_P>0\) on \(S^2\), contradicting
Proposition~\ref{prop-no-positive-homogeneous-jacobian-R3}.
\end{proof}

We next isolate the nodal-set input used in the proof of the homogeneous
non-injectivity theorem.

\begin{lemma}
\label{lem:critical-zero-nodal-crossing}
Let \(u\) be a nontrivial real spherical harmonic on \(S^2\). Suppose that
\[
        u(\theta_0)=0,
        \qquad
        \nabla_{S^2}u(\theta_0)=0
\]
for some \(\theta_0\in S^2\). Then the nodal set
\[
        Z(u)=\{\theta\in S^2:u(\theta)=0\}
\]
is not locally homeomorphic to an open interval at \(\theta_0\). More
precisely, there exists an integer \(q\ge2\) such that, in local conformal
coordinates \(z=x+iy\) centered at \(\theta_0\),
\[
        u(x,y)=\operatorname{Re}(\alpha z^q)+O(|z|^{q+1}),
        \qquad
        \alpha\in\C\setminus\{0\}.
\]
Consequently, \(Z(u)\) consists locally of \(2q\ge4\) analytic half-arcs
issuing from \(\theta_0\), with equal angles.
\end{lemma}

\begin{proof}
Since \(u\) is a spherical harmonic of some degree \(m\), it satisfies
\[
        \Delta_{S^2}u+m(m+1)u=0.
\]
Choose local conformal coordinates \(z=x+iy\) centered at \(\theta_0\). In
these coordinates the spherical metric has the form
\[
        \mathrm{d}s^2=e^{2\varphi(x,y)}(\mathrm{d}x^2+\mathrm{d}y^2),
\]
and hence the eigenvalue equation becomes
\[
        \Delta u+m(m+1)e^{2\varphi(x,y)}u=0
\]
for a smooth real-valued function \(\varphi\).

Since \(u(\theta_0)=0\) and \(\nabla_{S^2}u(\theta_0)=0\), the order of
vanishing of \(u\) at \(\theta_0\) is some integer \(q\ge2\). Thus,
\[
        u(x,y)=H_q(x,y)+O(|z|^{q+1}),
\]
where \(H_q\) is a nonzero homogeneous polynomial of degree \(q\). Comparing
the lowest-order terms in the preceding elliptic equation gives
\[
        \Delta H_q=0.
\]
Therefore, \(H_q\) is a nonzero homogeneous harmonic polynomial in two
variables. Hence, there exists \(\alpha\in\C\setminus\{0\}\) such that
\[
        H_q(x,y)=\operatorname{Re}(\alpha z^q).
\]
Thus,
\[
        u(x,y)=\operatorname{Re}(\alpha z^q)+O(|z|^{q+1}).
\]

The standard local structure theorem for nodal sets of eigenfunctions on
surfaces now implies that the nodal set \(Z(u)\) consists locally of \(2q\)
analytic half-arcs issuing from \(\theta_0\), with equal angles, see
Cheng~\cite{Cheng1976}. Since \(q\ge2\), at least four half-arcs meet at
\(\theta_0\). Hence, \(Z(u)\) is not locally homeomorphic to an open interval
at \(\theta_0\).
\end{proof}

\begin{proof}[Proof of Theorem~\ref{thm-homogeneous-noninjectivity}]
Suppose, to the contrary, that
\[
        P:\R^3\to\R^3
\]
is a one-to-one homogeneous harmonic polynomial mapping of degree \(m>1\).

If \(m\) is even, then
\[
        P(-x)=P(x),
\]
and injectivity is impossible. Hence, \(m\) is odd. Since \(P\) is homogeneous
and one-to-one, we have
\[
        P^{-1}(0)=\{0\}.
\]
Therefore, the radial normalization
\[
        F:S^2\to S^2,
        \qquad
        F(\theta)=\frac{P(\theta)}{|P(\theta)|},
\]
is well-defined.

We first show that \(F\) is one-to-one. Indeed, if \(F(\theta_1)=F(\theta_2)\),
then there exists \(c>0\) such that
\[
        P(\theta_1)=cP(\theta_2).
\]
Choosing \(s=c^{1/m}>0\), homogeneity gives
\[
        P(\theta_1)=P(s\theta_2).
\]
Since \(P\) is one-to-one, \(\theta_1=s\theta_2\). As
\(\theta_1,\theta_2\in S^2\), it follows that \(s=1\), and hence
\[
        \theta_1=\theta_2.
\]
Thus, \(F\) is injective. Since \(S^2\) is compact, \(F\) is a homeomorphism
onto its image. By invariance of domain, \(F(S^2)\) is open in \(S^2\).
It is also compact, hence closed. Since \(S^2\) is connected and
\(F(S^2)\neq\varnothing\), we obtain
\[
        F(S^2)=S^2.
\]
Thus, \(F\) is a homeomorphism of \(S^2\).

We claim next that
\[
        J_P(\theta)\neq0,
        \qquad
        \theta\in S^2.
\]
Assume otherwise that
\[
        J_P(\theta_0)=0
\]
for some \(\theta_0\in S^2\). Choose an oriented orthonormal basis
\(\tau_1,\tau_2\) of \(T_{\theta_0}S^2\). By the spherical Jacobian identity,
\[
        J_P(\theta_0)
        =
        m\det\bigl(\mathrm{d}P_{\theta_0}\tau_1,
                   \mathrm{d}P_{\theta_0}\tau_2,
                   P(\theta_0)\bigr).
\]
Hence, the three vectors
\[
        \mathrm{d}P_{\theta_0}\tau_1,\qquad
        \mathrm{d}P_{\theta_0}\tau_2,\qquad
        P(\theta_0)
\]
are linearly dependent. Therefore, there exists a nonzero vector
\(a\in\R^3\) such that
\[
        a\cdot P(\theta_0)=0,
        \qquad
        a\cdot \mathrm{d}P_{\theta_0}\tau_1=0,
        \qquad
        a\cdot \mathrm{d}P_{\theta_0}\tau_2=0.
\]
Define
\[
        u(\theta)=a\cdot P(\theta),
        \qquad
        \theta\in S^2.
\]
Then \(u\) is a spherical harmonic of degree \(m\).

We first note that \(u\not\equiv0\). Indeed, if \(u\equiv0\), then
\[
        P(S^2)\subset a^\perp.
\]
Consequently,
\[
        F(S^2)
        =
        \left\{\frac{P(\theta)}{|P(\theta)|}:\theta\in S^2\right\}
        \subset S^2\cap a^\perp,
\]
which contradicts the fact that \(F\) is onto \(S^2\). Thus, \(u\) is
nontrivial.

Moreover,
\[
        u(\theta_0)=0,
        \qquad
        \nabla_{S^2}u(\theta_0)=0.
\]
By Lemma~\ref{lem:critical-zero-nodal-crossing}, the nodal set
\[
        Z(u)=\{\theta\in S^2:u(\theta)=0\}
\]
is not locally homeomorphic to an open interval at \(\theta_0\).

On the other hand, since \(F\) is a homeomorphism,
\[
        Z(u)
        =
        \{\theta\in S^2:a\cdot P(\theta)=0\}
        =
        F^{-1}(S^2\cap a^\perp).
\]
The set \(S^2\cap a^\perp\) is a great circle, hence a topological circle.
Therefore, \(Z(u)\) is homeomorphic to a circle. In particular, it is locally
homeomorphic to an open interval at every point. This contradicts the
conclusion of Lemma~\ref{lem:critical-zero-nodal-crossing}.

Thus,
\[
        J_P(\theta)\neq0,
        \qquad
        \theta\in S^2.
\]
Since \(S^2\) is connected, \(J_P\) has a constant sign on \(S^2\). Reversing
orientation in the target if necessary, we may assume
\[
        J_P(\theta)>0,
        \qquad
        \theta\in S^2.
\]
This contradicts Proposition~\ref{prop-no-positive-homogeneous-jacobian-R3}.
Therefore, no one-to-one homogeneous harmonic polynomial mapping
\[
        P:\R^3\to\R^3
\]
of degree \(m>1\) exists.
\end{proof}

\section{The Jacobian nonvanishing theorem and global consequences}
\label{sec-lewy}

\begin{proof}[Proof of Theorem~\ref{thm-main}]
Let \(f:\Omega\to\R^3\) be a nonconstant sense-preserving quasiregular
harmonic mapping. Since the coordinate functions of \(f\) are harmonic,
\(f\) is real analytic. Suppose, to the contrary, that
\[
        J_f(a)=0
\]
for some \(a\in\Omega\). By Lemma~\ref{lem-pointwise-distortion}, the
distortion inequality holds pointwise for the smooth mapping \(f\). Hence,
\[
        \norm{Df(a)}^3\le KJ_f(a)=0,
\]
and therefore
\[
        Df(a)=0.
\]
By Lemma~\ref{lem-analytic-blow-up}, the first nonzero Taylor term of
\(f-f(a)\) at \(a\) is a nonzero homogeneous harmonic polynomial mapping
\[
        P_m:\R^3\to\R^3,\qquad m\ge2.
\]
Moreover, the rescaled mappings
\[
        f_r(x)=\frac{f(a+rx)-f(a)}{r^m}
\]
converge locally uniformly, in fact in \(C^\infty_{\operatorname{loc}}\), to
\(P_m\) as \(r\downarrow0\). Each \(f_r\) is \(K\)-quasiregular on its
rescaled domain, because domain dilations, target translations and positive
target homotheties preserve the quasiregular distortion inequality. Since
\(P_m\) is nonconstant, Reshetnyak compactness implies that \(P_m\) is
\(K\)-quasiregular. This contradicts
Theorem~\ref{thm-homogeneous-qr-obstruction}, which excludes nonconstant
quasiregular homogeneous harmonic polynomial mappings \(\R^3\to\R^3\) of
degree \(m>1\).

Thus, \(J_f\) has no zeros in \(\Omega\). Since \(f\) is sense-preserving
quasiregular, \(J_f\ge0\) almost everywhere. As \(J_f\) is continuous, this
implies \(J_f\ge0\) everywhere. Since \(J_f\) has no zeros, we obtain
\[
        J_f(x)>0,
        \qquad
        x\in\Omega.
\]

If \(f:\Omega\to\Omega'\) is a quasiconformal harmonic mapping which is
orientation-reversing, compose it with a fixed orientation-reversing linear
isometry of the target. The resulting mapping is sense-preserving, harmonic
and quasiconformal, so the previous argument gives nonvanishing Jacobian for
the reflected mapping. Hence, the original mapping also satisfies \(J_f\ne0\).
The inverse function theorem then shows that \(f\) is a local harmonic
diffeomorphism.
\end{proof}

\begin{corollary}
Let \(f:\Omega\subset\R^3\to\R^3\) be a nonconstant sense-preserving
quasiregular harmonic mapping. Then the branch set of \(f\) is empty and, more
strongly, \(Df\) is invertible at every point.
\end{corollary}

\begin{proof}
By Theorem~\ref{thm-main}, \(J_f>0\) everywhere. Hence, \(Df\) is invertible
at every point, and the conclusion follows from the inverse function theorem.
The topological local-homeomorphism conclusion also follows from the general
smooth quasiregular local-invertibility theorem cited in the introduction. The
additional content here is the nonvanishing of the Jacobian.
\end{proof}

\begin{remark}
The bounded-distortion assumption in Theorem~\ref{thm-main} is essential to
the statement. Wood \cite{Wood1991} constructed a harmonic homeomorphism in
higher dimension whose Jacobian vanishes at an interior point. Thus,
harmonicity and topological injectivity alone do not force local
diffeomorphism once one leaves the planar setting. The point of
Theorem~\ref{thm-main} is that, in dimension three, the quasiregular
distortion inequality rules out precisely the homogeneous critical blow-ups
that could occur at a zero of the Jacobian.
\end{remark}

\subsection{Riemannian variants}

The preceding proof is local and depends only on the Euclidean principal part
of the harmonic map equation. We record the corresponding Riemannian form,
which is useful for separating the geometric content of the argument from the
choice of Euclidean coordinates. In this subsection, a sense-preserving
\(K\)-quasiregular mapping between oriented Riemannian three-manifolds means a
nonconstant mapping satisfying
\[
        J_f^{g,h}\ge 0\quad\text{a.e.},
        \qquad
        \|\mathrm{d}f\|_{g,h}^3\le KJ_f^{g,h}\quad\text{a.e.},
\]
where \(J_f^{g,h}\) is the Riemannian Jacobian determined by the metrics and
orientations.

\begin{corollary}
\label{cor-riemannian-lewy}
Let \((M^3,g)\) and \((N^3,h)\) be connected oriented real-analytic
Riemannian three-manifolds. Let
\[
        f:(M^3,g)\to (N^3,h)
\]
be a nonconstant sense-preserving \(K\)-quasiregular mapping which is a
classical real-analytic harmonic map. Then its Riemannian Jacobian satisfies
\[
        J_f^{g,h}(p)>0,
        \qquad p\in M.
\]
Consequently, \(f\) is a local real-analytic diffeomorphism.
\end{corollary}

\begin{proof}
The assertion is local. Fix \(a\in M\). Choose oriented normal coordinates
\(x=(x^1,x^2,x^3)\) for \(g\) centered at \(a\), and oriented normal
coordinates \(y=(y^1,y^2,y^3)\) for \(h\) centered at \(f(a)\). In these
coordinates we write again
\[
        f(x)=(f^1(x),f^2(x),f^3(x)),
        \qquad f(0)=0.
\]
Assume, to the contrary, that
\[
        J_f^{g,h}(a)=0.
\]
Since \(f\) is smooth and quasiregular, the distortion inequality holds
pointwise by continuity, and hence
\[
        \|\mathrm{d}f_a\|_{g,h}^3\le KJ_f^{g,h}(a)=0.
\]
Therefore, \(\mathrm{d}f_a=0\).

Since \(f\) is real analytic and nonconstant, there is a first nonzero
homogeneous Taylor term
\[
        f(x)=P_m(x)+O(|x|^{m+1}),
        \qquad m\ge2,
\]
where \(P_m:\R^3\to\R^3\) is a nonzero homogeneous polynomial mapping of
degree \(m\). We claim that \(P_m\) is harmonic componentwise with respect to
the Euclidean Laplacian.

In local coordinates, the harmonic map equation is
\[
g^{\alpha\beta}(x)
\left(
\partial_{\alpha\beta}f^i
-\Gamma^\gamma_{\alpha\beta}(x)\partial_\gamma f^i
+\widetilde\Gamma^i_{jk}(f(x))\partial_\alpha f^j\partial_\beta f^k
\right)=0,
\qquad i=1,2,3.
\]
In the chosen normal coordinates,
\[
        g^{\alpha\beta}(x)=\delta^{\alpha\beta}+O(|x|^2),
        \qquad
        \Gamma^\gamma_{\alpha\beta}(x)=O(|x|),
        \qquad
        \widetilde\Gamma^i_{jk}(y)=O(|y|).
\]
The lowest-order term in the equation is therefore
\[
        \delta^{\alpha\beta}\partial_{\alpha\beta}P_m^i
        =\Delta P_m^i,
\]
which is homogeneous of degree \(m-2\). All connection terms have strictly
higher order because \(m\ge2\). Hence,
\[
        \Delta P_m^i=0,
        \qquad i=1,2,3.
\]
Thus, \(P_m\) is a homogeneous Euclidean harmonic polynomial mapping.

Define the rescalings
\[
        f_r(x)=\frac{f(rx)}{r^m}.
\]
Then \(f_r\to P_m\) locally uniformly, in fact in
\(C^\infty_{\operatorname{loc}}\). Under this rescaling, the domain metrics
\(g(rx)\) and the target metrics \(h(r^m y)\), up to harmless constant
homotheties, converge in \(C^\infty_{\operatorname{loc}}\) to the Euclidean
metric. Consequently, on every fixed compact subset, the Riemannian
quasiregular inequality for \(f\) gives a uniform Euclidean quasiregular
distortion bound for \(f_r\), for all sufficiently small \(r\). Passing to the
locally uniform limit and using Reshetnyak compactness, we obtain that
\[
        P_m:\R^3\to\R^3
\]
is a nonconstant Euclidean quasiregular mapping.

This contradicts Theorem~\ref{thm-homogeneous-qr-obstruction}, since \(P_m\)
is homogeneous, Euclidean harmonic, and has degree \(m\ge2\). Therefore, the
Riemannian Jacobian cannot vanish. Since \(f\) is sense-preserving, continuity
and the a.e. inequality \(J_f^{g,h}\ge0\) give \(J_f^{g,h}\ge0\) everywhere.
The absence of zeros then gives \(J_f^{g,h}>0\) everywhere. The inverse
function theorem gives the local real-analytic diffeomorphism conclusion.
\end{proof}

\begin{corollary}
\label{cor-hyperbolic-lewy}
Let \(\Omega\subset\B^3\) be a domain equipped with the Poincar\'e metric
\(g_{\mathbb H}\), and let
\[
        f:\Omega\to\R^3
\]
be a nonconstant sense-preserving quasiregular mapping from
\((\Omega,g_{\mathbb H})\) to Euclidean \(\R^3\). Assume that the coordinate
functions are hyperbolic-harmonic,
\[
        \Delta_{g_{\mathbb H}}f^i=0,
        \qquad i=1,2,3.
\]
Then
\[
        J_f(x)>0,
        \qquad x\in\Omega,
\]
where \(J_f\) may be taken either as the Riemannian Jacobian or,
equivalently up to a positive conformal factor, as the Euclidean Jacobian.
Consequently, \(f\) is a local real-analytic diffeomorphism.
\end{corollary}

\begin{proof}
This is Corollary~\ref{cor-riemannian-lewy} with domain metric equal to the
hyperbolic metric and target metric Euclidean. Equivalently, in local
coordinates the hyperbolic Laplacian has Euclidean principal part, and the
first nonzero homogeneous blow-up at a hypothetical zero of the Jacobian is a
Euclidean homogeneous harmonic polynomial mapping. The homogeneous quasiregular
obstruction applies as before.
\end{proof}

\begin{remark}
The real-analytic formulation is chosen to keep the statement directly
compatible with the elementary Taylor blow-up used throughout the paper. For
smooth non-analytic metrics one expects the same conclusion under suitable
finite-order vanishing and tangent-map results for elliptic systems, but that
extension is not needed for the present work.
\end{remark}

\subsection{Global rigidity consequences}

The local theorem has a useful global consequence at infinity: there are no
nonlinear entire quasiregular harmonic mappings in dimension three. The proof
has three ingredients. First, Theorem~\ref{thm-main} gives \(J_f>0\), and hence removes the branch set.
Second, Zorich's global homeomorphism theorem promotes the mapping to a
quasiconformal automorphism of \(\R^3\). Third, Mori's growth theorem gives
polynomial growth, so that the harmonic coordinate functions are actually
harmonic polynomials. The leading homogeneous term at infinity is then
excluded by Theorem~\ref{thm-homogeneous-qr-obstruction} unless it has degree
one.

\begin{lemma}
\label{lem-zorich}
Let \(n\ge3\), and let \(F:\R^n\to\R^n\) be a nonconstant quasiregular
mapping with empty branch set. Then \(F\) is a quasiconformal automorphism of
\(\R^n\).
\end{lemma}

\begin{proof}
By Reshetnyak's theorem, \(F\) is open and discrete. Since the branch set is
empty, \(F\) is a local homeomorphism. Zorich's global homeomorphism theorem
states that every locally homeomorphic quasiregular mapping
\[
        \R^n\to\R^n,\qquad n\ge3,
\]
is a homeomorphism of \(\R^n\), see \cite{Zorich1967} and also
\cite[Section II.6]{Rickman} or \cite[Section 17]{Vaisala}. Hence, \(F\) is a
homeomorphism of \(\R^n\) onto \(\R^n\). A one-to-one quasiregular mapping is
quasiconformal. Therefore, \(F\) is a quasiconformal automorphism of \(\R^n\).
\end{proof}

\begin{lemma}
\label{lem-polynomial-growth}
Let \(F:\R^3\to\R^3\) be a quasiconformal automorphism. Then there are
constants \(C<\infty\) and \(N<\infty\) such that
\[
        |F(x)|\le C(1+|x|)^N,
        \qquad
        x\in\R^3.
\]
\end{lemma}

\begin{proof}
After replacing \(F\) by \(F-F(0)\), assume \(F(0)=0\). A quasiconformal
automorphism of \(\R^3\) extends to a quasiconformal self-homeomorphism of
the one-point compactification
\[
        S^3=\R^3\cup\{\infty\},
\]
fixing \(\infty\). Mori's theorem, or equivalently the standard H\"older
distortion estimate for quasiconformal mappings in the spherical metric
\cite{Vaisala}, gives a power bound at \(\infty\). In Euclidean form this
yields, for \(|x|\ge1\),
\[
        |F(x)|\le C|x|^N,
\]
where \(N\) depends only on the distortion of \(F\). Enlarging \(C\) gives
the stated estimate on all of \(\R^3\).
\end{proof}

\begin{corollary}
\label{cor-entire-liouville}
Let \(f:\R^3\to\R^3\) be a nonconstant sense-preserving quasiregular harmonic
mapping. Then \(f\) is affine:
\[
        f(x)=Ax+b,
        \qquad
        \det A>0.
\]
In particular, after possibly reversing the orientation in the target, every
entire quasiregular harmonic mapping \(\R^3\to\R^3\) is affine.
\end{corollary}

\begin{proof}
By Theorem~\ref{thm-main}, \(J_f>0\) everywhere. Hence, \(f\) has empty
branch set. By Lemma~\ref{lem-zorich}, \(f\) is a quasiconformal
automorphism of \(\R^3\). Therefore, Lemma~\ref{lem-polynomial-growth} gives
polynomial growth:
\[
        |f(x)|\le C(1+|x|)^N.
\]
Each coordinate function of \(f\) is an entire harmonic function of
polynomial growth. By the classical Liouville theorem for harmonic
functions, each coordinate is a harmonic polynomial.

Let \(\ell\) be the degree of the resulting polynomial mapping and write
\[
        f(x)=P_\ell(x)+P_{\ell-1}(x)+\cdots+P_0,
\]
where \(P_j\) is homogeneous of degree \(j\), and \(P_\ell\not\equiv0\). Since
the coordinate functions of \(f\) are harmonic, each homogeneous term is
harmonic. If \(\ell>1\), then the blow-downs
\[
        f_R(x)=\frac{f(Rx)}{R^\ell}
\]
are quasiregular with the same distortion as \(f\), and \(f_R\to P_\ell\)
locally uniformly as \(R\to\infty\). Reshetnyak compactness therefore
implies that \(P_\ell\) is a nonconstant quasiregular homogeneous harmonic
polynomial mapping of degree \(\ell>1\). This contradicts
Theorem~\ref{thm-homogeneous-qr-obstruction}. Hence, \(\ell=1\), and
\(f(x)=Ax+b\).

Finally, the quasiregular distortion inequality excludes a singular
nonconstant affine part. Hence, \(\det A>0\) in the sense-preserving case.
\end{proof}

\begin{corollary}
Let \(F:\R^3\to\R^3\) be a nonconstant sense-preserving quasiregular mapping
whose coordinate functions are harmonic polynomials. Then \(F\) is affine.
\end{corollary}

\begin{proof}
This is the last blow-down part of the proof of
Corollary~\ref{cor-entire-liouville}, without using Lemma~\ref{lem-zorich}
or Lemma~\ref{lem-polynomial-growth}.
\end{proof}

\begin{remark}
The quasiregular bounded-distortion assumption in
Corollary~\ref{cor-entire-liouville} is essential. There are nonlinear entire
harmonic diffeomorphisms of \(\R^3\) onto itself. For example,
\[
        F(x,y,z)=(x,y,z+x^2-y^2).
\]
Its coordinate functions are harmonic, since
\[
        \Delta x=0,
        \qquad
        \Delta y=0,
        \qquad
        \Delta(z+x^2-y^2)=0.
\]
Moreover,
\[
        DF=
        \begin{pmatrix}
        1 & 0 & 0\\
        0 & 1 & 0\\
        2x & -2y & 1
        \end{pmatrix},
        \qquad
        J_F\equiv1.
\]
Thus, \(F\) is a global real-analytic diffeomorphism, with inverse
\[
        F^{-1}(X,Y,Z)=(X,Y,Z-X^2+Y^2).
\]
However, \(F\) is not quasiregular, because
\(\norm{DF(x,y,z)}\to\infty\) as \(|(x,y)|\to\infty\), while \(J_F\equiv1\).
Hence, no constant \(K\) can satisfy
\[
        \norm{DF}^3\le KJ_F
\]
on all of \(\R^3\).
\end{remark}

\section{Sharpness and higher-dimensional limitations}
\label{sec-higher}

The dimension-three theorem suggests a natural homogeneous obstruction
problem in arbitrary dimension. Before passing to higher dimensions, we first
record a sharp borderline cubic in \(\R^3\). It shows that the positivity
assumption in Proposition~\ref{prop-no-positive-homogeneous-jacobian-R3} cannot
be relaxed to nonnegativity.

\subsection{A sharp borderline cubic in dimension three}

\begin{proof}[Proof of Theorem~\ref{thm-icosahedral-borderline}]
Define
\[
\begin{aligned}
        Q_1(x,y,z)&=x^3-3xy^2+3z(x^2-y^2),\\
        Q_2(x,y,z)&=6xyz-3x^2y+y^3,\\
        Q_3(x,y,z)&=z(2z^2-3x^2-3y^2),
\end{aligned}
\]
and put \(Q=(Q_1,Q_2,Q_3)\). These three components are homogeneous harmonic
cubics. Indeed, they are the standard cubic spherical harmonics
\[
        Y_1=x^3-3xy^2,\quad
        Y_2=3x^2y-y^3,\quad
        Y_5=z(x^2-y^2),\quad
        Y_6=2xyz,\quad
        Y_7=z(2z^2-3x^2-3y^2)
\]
combined as
\[
        Q=(Y_1+3Y_5,\,3Y_6-Y_2,\,Y_7).
\]

Let
\[
        v_0=(0,0,1),
\]
and, for \(k=1,\ldots,5\), let
\[
        v_k=\left(
        \frac2{\sqrt5}\cos\frac{2\pi(k-1)}5,
        \frac2{\sqrt5}\sin\frac{2\pi(k-1)}5,
        \frac1{\sqrt5}
        \right).
\]
The twelve points \(\pm v_0,\ldots,\pm v_5\) are the vertices of a regular
icosahedron; see Figure~\ref{fig:icosahedral-vertices}.

\begin{figure}[htbp]
\centering
\includegraphics[width=.56\textwidth]{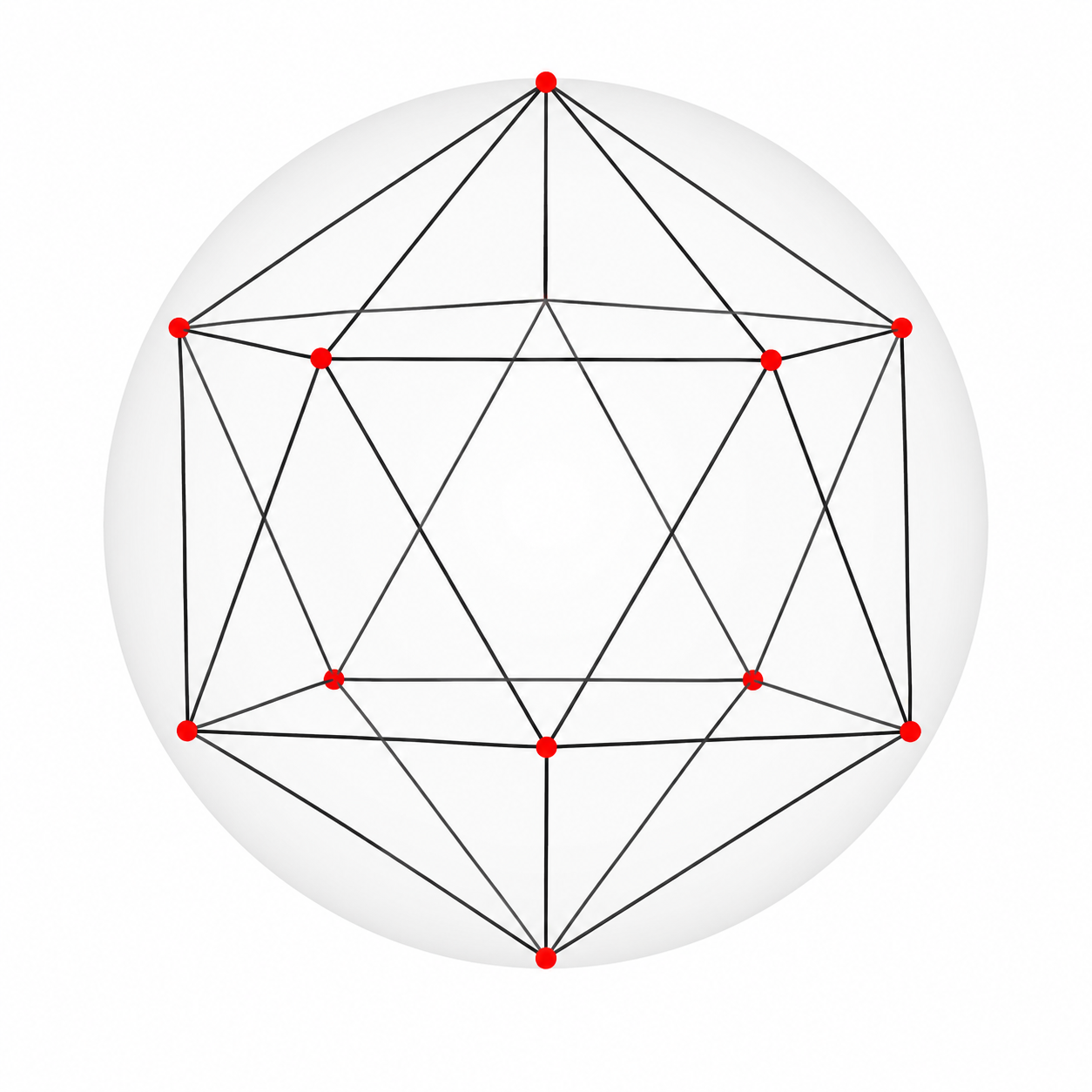}
\caption{The vertex configuration \(\{\pm v_0,\ldots,\pm v_5\}\) on the unit sphere used in the icosahedral cubic.}
\label{fig:icosahedral-vertices}
\end{figure}

A direct determinant computation gives the identity
\[
        J_Q(X)=\frac{225}{2}
        \left(
        \frac{26}{25}|X|^6-
        \sum_{k=0}^{5}(v_k\cdot X)^6
        \right),
        \qquad X\in\R^3.
\]
It remains to prove the elementary extremal inequality
\[
        \sum_{k=0}^{5}(v_k\cdot\theta)^6\le \frac{26}{25},
        \qquad \theta\in S^2,
\]
and to identify the equality cases.

Write
\[
        \theta=(r\cos\phi,r\sin\phi,t),
        \qquad r\ge0,
        \qquad r^2+t^2=1.
\]
Using the elementary trigonometric sums over the five fifth-roots of unity,
one obtains
\[
\begin{aligned}
        \sum_{k=0}^{5}(v_k\cdot\theta)^6
        =\frac2{25}\bigl(&10r^6+6r^5t\cos5\phi+45r^4t^2 \\
        &+15r^2t^4+13t^6\bigr).
\end{aligned}
\]
Put \(u=t^2\). Since \(|\cos5\phi|\le1\), the last display gives
\[
\begin{aligned}
        \frac{25}{2}
        \sum_{k=0}^{5}(v_k\cdot\theta)^6
        \le{}&10(1-u)^3+6(1-u)^{5/2}\sqrt u \\
        &+45(1-u)^2u+15(1-u)u^2+13u^3.
\end{aligned}
\]
Set
\[
        A(u)=10(1-u)^3+45(1-u)^2u+15(1-u)u^2+13u^3.
\]
Then
\[
        13-A(u)=3(1-u)(11u^2-4u+1)\ge0
        \qquad (0\le u\le1),
\]
and a direct simplification gives
\[
\begin{aligned}
        \bigl(13-A(u)\bigr)^2-36u(1-u)^5
        ={}&9(1-u)^2(5u-1)^2(5u^2-2u+1)\ge0.
\end{aligned}
\]
Therefore,
\[
        A(u)+6(1-u)^{5/2}\sqrt u\le13,
\]
and hence \(\sum_{k=0}^{5}(v_k\cdot\theta)^6\le26/25\). It follows that
\(J_Q\ge0\) on \(S^2\), and by homogeneity on all of \(\R^3\).

The equality conditions in the preceding inequalities are also explicit. One
gets either \(u=1\), which gives \(\theta=\pm v_0\), or \(u=1/5\) and
\(\cos5\phi=\operatorname{sgn}t\), which gives the remaining vertices
\(\pm v_1,\ldots,\pm v_5\). Thus, the zero set of \(J_Q\) on \(S^2\) is
exactly the icosahedral vertex set. Finally,
\[
        DQ(v_0)v_0=(0,0,6)\ne0,
        \qquad
        J_Q(v_0)=0.
\]
Hence, the quasiregular distortion inequality cannot hold at \(v_0\). Thus,
\(Q\) is not quasiregular.
\end{proof}

\begin{remark}
Theorem~\ref{thm-icosahedral-borderline} shows that the three-dimensional
minimum-principle obstruction is exactly a strict positivity obstruction. It proves that
one cannot hope to strengthen Proposition~\ref{prop-no-positive-homogeneous-jacobian-R3}
to say that every homogeneous harmonic Jacobian of degree \(m>1\) must change
sign on \(S^2\). The correct conclusion is that it cannot be strictly
positive or strictly negative. Nonnegative borderline cubics do exist, but at
their zeros the differential need not vanish, and bounded distortion fails.
\end{remark}

\subsection{Dimension-free reductions}

We record the dimension-free parts.

\begin{proposition}
Let \(n\ge2\), and let \(P:\R^n\to\R^n\) be a homogeneous harmonic
polynomial mapping of degree \(m\ge1\). If
\[
        J_P(\theta)\ne0,
        \qquad
        \theta\in S^{n-1},
\]
then \(P(\theta)\ne0\) on \(S^{n-1}\), and the radial normalization
\[
        \Phi(\theta)=\frac{P(\theta)}{|P(\theta)|}
\]
is a local diffeomorphism of \(S^{n-1}\). If \(n\ge3\), then \(\Phi\) is a
diffeomorphism and \(m\) is odd.
\end{proposition}

\begin{proof}
If \(P(\theta)=0\) at some point of the sphere, Euler's identity would make
the radial column \(DP(\theta)\theta=mP(\theta)\) vanish, and hence
\(J_P(\theta)=0\), a contradiction. Thus, \(P\ne0\) on the sphere.

Let \(X_1,\ldots,X_{n-1}\) be an oriented orthonormal frame at \(\theta\). By
Lemma~\ref{lem-spherical-jacobian}, the vectors
\[
        P(\theta),\ \mathrm{d}P_\theta X_1,\ldots,\mathrm{d}P_\theta X_{n-1}
\]
span \(\R^n\). Radial projection from \(\R^n\setminus\{0\}\) onto the sphere
kills only the radial direction. Hence, the derivatives of \(\Phi\) span
\(T_{\Phi(\theta)}S^{n-1}\), and \(\Phi\) is a local diffeomorphism.

For \(n\ge3\), the sphere \(S^{n-1}\) is compact, connected, and simply
connected. The image of a local diffeomorphism from a compact manifold is
open and closed, so \(\Phi\) is onto. It is therefore a covering map of
\(S^{n-1}\). Since \(S^{n-1}\) is simply connected, the covering has one
sheet, and \(\Phi\) is a diffeomorphism. If \(m\) were even, then
\(P(-\theta)=P(\theta)\), so \(\Phi(-\theta)=\Phi(\theta)\), contradicting
injectivity. Hence, \(m\) is odd.
\end{proof}

The next reformulation uses only the standard relation between homogeneous
harmonic polynomials and spherical harmonics, as presented for instance in
\cite{AxlerBourdonRamey}. We shall also use the usual first-jet terminology:
the first jet of a smooth function at a point consists of its value and first
differential, see, for example, \cite{Saunders} for the jet-bundle language.

\begin{proposition}
Let \(m\ge1\), and let \(\mathcal H_m(S^{n-1})\) denote the space of
restrictions to the sphere of homogeneous harmonic polynomials of degree
\(m\) on \(\R^n\). There exists a homogeneous harmonic polynomial mapping
\(P:\R^n\to\R^n\) of degree \(m\) with \(J_P\ne0\) on \(S^{n-1}\) if and only
if there exists an \(n\)-dimensional subspace
\(W\subset\mathcal H_m(S^{n-1})\) such that the first-jet evaluation map
\[
        j^1_\theta:W\longrightarrow\R\oplus T^*_\theta S^{n-1},
        \qquad
        j^1_\theta u=(u(\theta),\mathrm{d}u_\theta),
\]
is an isomorphism for every \(\theta\in S^{n-1}\). Here \(j^1_\theta u\)
denotes the first jet of \(u\) at \(\theta\), namely its value and first
differential at that point.
\end{proposition}

\begin{proof}
We first assume that such a homogeneous harmonic polynomial mapping
\[
        P=(P^1,\ldots,P^n):\R^n\to\R^n
\]
is given and that
\[
        J_P(\theta)\ne0,
        \qquad
        \theta\in S^{n-1}.
\]
Set
\[
        W=\spanop\{P^1|_{S^{n-1}},\ldots,P^n|_{S^{n-1}}\}
        \subset \mathcal H_m(S^{n-1}).
\]
We first note that \(\dim W=n\). Indeed, if the restrictions of the
components were linearly dependent, then there would exist a nonzero vector
\(a=(a_1,\ldots,a_n)\in\R^n\) such that
\[
        a\cdot P(\theta)=0,
        \qquad
        \theta\in S^{n-1}.
\]
Thus, \(P(S^{n-1})\subset a^\perp\). Differentiating along the sphere also
gives
\[
        a\cdot \mathrm{d}P_\theta X=0,
        \qquad
        X\in T_\theta S^{n-1}.
\]
Hence, for every \(\theta\in S^{n-1}\), all the vectors
\[
        P(\theta),\ \mathrm{d}P_\theta X_1,\ldots,\mathrm{d}P_\theta X_{n-1}
\]
belong to the hyperplane \(a^\perp\), and therefore cannot span \(\R^n\).
This contradicts Lemma~\ref{lem-spherical-jacobian}, since \(J_P(\theta)\ne0\).
Thus, \(\dim W=n\).

Fix \(\theta\in S^{n-1}\), and let
\[
        X_1,\ldots,X_{n-1}
\]
be a basis of \(T_\theta S^{n-1}\). With respect to the basis
\[
        P^1|_{S^{n-1}},\ldots,P^n|_{S^{n-1}}
\]
of \(W\), and the basis of \(\R\oplus T^*_\theta S^{n-1}\) given by
evaluation and differentiation in the directions \(X_1,\ldots,X_{n-1}\), the
matrix of \(j^1_\theta\) is
\[
        M(\theta)=
        \begin{pmatrix}
        P^1(\theta) & P^2(\theta) & \cdots & P^n(\theta)\\
        \mathrm{d}P^1_\theta X_1 & \mathrm{d}P^2_\theta X_1 & \cdots & \mathrm{d}P^n_\theta X_1\\
        \vdots & \vdots & \ddots & \vdots\\
        \mathrm{d}P^1_\theta X_{n-1} & \mathrm{d}P^2_\theta X_{n-1}
        & \cdots & \mathrm{d}P^n_\theta X_{n-1}
        \end{pmatrix}.
\]
This is the transpose of the matrix whose columns are
\[
        P(\theta),\ \mathrm{d}P_\theta X_1,\ldots,\mathrm{d}P_\theta X_{n-1}.
\]
Therefore, \(M(\theta)\) is invertible if and only if these \(n\) vectors span
\(\R^n\). By Lemma~\ref{lem-spherical-jacobian}, this is equivalent to
\[
        J_P(\theta)\ne0.
\]
Hence, \(j^1_\theta:W\to\R\oplus T^*_\theta S^{n-1}\) is an isomorphism for
every \(\theta\in S^{n-1}\).

Conversely, suppose that \(W\subset\mathcal H_m(S^{n-1})\) is an
\(n\)-dimensional subspace such that
\[
        j^1_\theta:W\to\R\oplus T^*_\theta S^{n-1}
\]
is an isomorphism for every \(\theta\in S^{n-1}\). Choose a basis
\[
        u_1,\ldots,u_n
\]
of \(W\). By the definition of \(\mathcal H_m(S^{n-1})\), each \(u_i\) is the
restriction to \(S^{n-1}\) of a unique homogeneous harmonic polynomial
\(U_i\) of degree \(m\) on \(\R^n\). Define
\[
        P=(U_1,\ldots,U_n):\R^n\to\R^n.
\]
Then \(P\) is a homogeneous harmonic polynomial mapping of degree \(m\).

Fix \(\theta\in S^{n-1}\) and a basis
\(X_1,\ldots,X_{n-1}\) of \(T_\theta S^{n-1}\). The matrix of \(j^1_\theta\)
in the basis \(u_1,\ldots,u_n\) is again the transpose of the matrix whose
columns are
\[
        P(\theta),\ \mathrm{d}P_\theta X_1,\ldots,\mathrm{d}P_\theta X_{n-1}.
\]
Since \(j^1_\theta\) is an isomorphism, this matrix is invertible. Therefore,
the vectors
\[
        P(\theta),\ \mathrm{d}P_\theta X_1,\ldots,\mathrm{d}P_\theta X_{n-1}
\]
span \(\R^n\). By Lemma~\ref{lem-spherical-jacobian}, we obtain
\[
        J_P(\theta)\ne0.
\]
Since \(\theta\in S^{n-1}\) was arbitrary, \(J_P\ne0\) on \(S^{n-1}\).
\end{proof}

\begin{remark}
The trace identity of Proposition~\ref{prop-trace-formula} is
dimension-free, but its sign content is special to tangent dimension two. At
a positive local minimum of \(J_P|_{S^{n-1}}\), the first-derivative equations
give \(T_\gamma=0\). When \(n=3\), the spherical harmonic trace constraints
then force \(Q(H)\ge0\), and therefore \(\Delta J_P<0\), a contradiction. For
\(n\ge4\), the corresponding quadratic form is indefinite on the formal
second-jet constraints. For example, in a three-dimensional tangent block one
may take
\[
        H^1_{23}=H^1_{32}=t,
        \qquad
        H^2_{13}=H^2_{31}=-t,
\]
with the other tangent-component second derivatives equal to zero, obtaining
\[
        Q(H)=-2t^2.
\]
Thus, the present minimum-principle calculation is exactly a
three-dimensional obstruction, even though the blow-up and first-jet
reductions are valid in every dimension.
\end{remark}

\subsection{Equivariant cubic models in higher dimensions}

The preceding remarks leave open the full higher-dimensional strict problem.
Nevertheless, the natural rotationally equivariant cubic ansatz can be treated
completely. The result below strengthens the borderline construction: it
classifies all \(O(d)\)-equivariant harmonic cubics on \(\R^d\times\R\), shows
exactly when their Jacobians are one-signed, and proves that this whole class
contains no strict positive-Jacobian counterexample.

\begin{theorem}
\label{thm-equivariant-cubic-classification}
Let \(d\ge2\), put \(n=d+1\), and write \(x=(y,t)\in\R^d\times\R\). Every
\(O(d)\)-equivariant homogeneous cubic polynomial mapping
\(P:\R^{d+1}\to\R^{d+1}\) has the form
\[
        P(y,t)=\bigl((a t^2+b|y|^2)y,\,c t^3+e t|y|^2\bigr)
\]
for some real constants \(a,b,c,e\). Such a map is harmonic if and only if
\[
        a+(d+2)b=0,
        \qquad
        3c+de=0.
\]
Consequently, every \(O(d)\)-equivariant homogeneous harmonic cubic is of the
form
\[
        P_{a,c}(y,t)=
        \left(
        a\left(t^2-\frac{|y|^2}{d+2}\right)y,
        c\left(t^3-\frac3d\,t|y|^2\right)
        \right).
\]
Its Jacobian is
\[
\begin{aligned}
        J_{P_{a,c}}(y,t)
        ={}&a^d c
        \left(t^2-\frac{|y|^2}{d+2}\right)^{d-1}  \\
        &\times
        \frac{3\bigl(d(d+2)t^4+6|y|^2t^2+3|y|^4\bigr)}{d(d+2)}.
\end{aligned}
\]
The final factor is strictly positive away from the origin. Hence, no
nontrivial map in this equivariant harmonic cubic class satisfies
\(J_P>0\) on \(S^d\). More precisely:
\begin{enumerate}[label=(\roman*)]
\item if \(n=d+1\) is odd and \(ac\ne0\), then \(J_{P_{a,c}}\) changes sign
on \(S^d\),
\item if \(n=d+1\) is even and \(a^d c>0\), then \(J_{P_{a,c}}\ge0\) on
\(\R^n\), but its zero set on \(S^d\) consists of the two latitude spheres
\[
        |y|^2=(d+2)t^2,
        \qquad |y|^2+t^2=1,
\]
\item every nontrivial member of this equivariant class either has
identically zero Jacobian or has a nonzero differential at some point where
its Jacobian vanishes. In particular, none is quasiregular, and none is
one-to-one when \(d\ge2\) and \(ac\ne0\).
\end{enumerate}
\end{theorem}

\begin{proof}
Let \(P=(V,W)\), where \(V:\R^d\times\R\to\R^d\) and
\(W:\R^d\times\R\to\R\). The \(O(d)\)-equivariance condition means
\[
        V(Ry,t)=RV(y,t),
        \qquad
        W(Ry,t)=W(y,t),
        \qquad R\in O(d).
\]
Since \(P\) is a homogeneous cubic polynomial, the invariant scalar component
must be a linear combination of \(t^3\) and \(t|y|^2\), and the equivariant
vector component must be \(y\) multiplied by an invariant homogeneous
quadratic. Hence,
\[
        P(y,t)=\bigl((a t^2+b|y|^2)y,\,c t^3+e t|y|^2\bigr).
\]

For \(1\le i\le d\),
\[
        \Delta(y_i t^2)=2y_i,
        \qquad
        \Delta(y_i|y|^2)=2(d+2)y_i.
\]
Thus, the vector component is harmonic if and only if \(a+(d+2)b=0\). For the
scalar component,
\[
        \Delta(t^3)=6t,
        \qquad
        \Delta(t|y|^2)=2dt,
\]
so harmonicity is equivalent to \(3c+de=0\). This proves the classification
and gives the displayed form of \(P_{a,c}\).

Put \(r=|y|\),
\[
        A(y,t)=a\left(t^2-\frac{r^2}{d+2}\right),
        \qquad
        B(y,t)=c\left(t^3-\frac3d\,tr^2\right).
\]
Then \(P_{a,c}(y,t)=(A(y,t)y,B(y,t))\). By \(O(d)\)-equivariance, at a point
with \(r>0\) the derivative has \(d-1\) tangential eigenvalues equal to
\(A\). On the two-dimensional plane spanned by the radial \(y\)-direction and
the \(t\)-axis, the derivative is represented by
\[
        \begin{pmatrix}
        A+rA_r & rA_t\\
        B_r & B_t
        \end{pmatrix}.
\]
For the normalized functions
\[
        A_0=t^2-\frac{r^2}{d+2},
        \qquad
        B_0=t^3-\frac3d\,tr^2,
\]
a direct calculation gives
\[
        (A_0+r(A_0)_r)(B_0)_t-r(A_0)_t(B_0)_r
        =
        \frac{3\bigl(d(d+2)t^4+6r^2t^2+3r^4\bigr)}{d(d+2)}.
\]
Multiplying by \(a c\) for the radial--\(t\) block and by the \(d-1\)
tangential eigenvalues \(aA_0\) gives the stated Jacobian formula. The
factor
\[
        d(d+2)t^4+6r^2t^2+3r^4
\]
is positive away from the origin.

The sign conclusions now follow from the single factor
\[
        \left(t^2-\frac{|y|^2}{d+2}\right)^{d-1}.
\]
If \(d\) is even, then \(d-1\) is odd, and this factor changes sign across
the cone \(|y|^2=(d+2)t^2\). If \(d\) is odd, then \(d-1\) is even, so the
Jacobian is one-signed whenever \(a^dc\) has a fixed sign, but it vanishes on
that cone. Intersecting the cone with \(S^d\) gives
\[
        t^2=\frac1{d+3},
        \qquad
        |y|^2=\frac{d+2}{d+3}.
\]

At points on this cone, the tangential eigenvalues vanish, but the radial--\(t\)
block has nonzero determinant if \(ac\ne0\). Hence, \(DP_{a,c}\ne0\) while
\(J_{P_{a,c}}=0\), so the quasiregular distortion inequality cannot hold.
If \(a=0\) or \(c=0\), the Jacobian is identically zero and the map is again
not a nonconstant quasiregular mapping. Finally, when \(ac\ne0\), the cone is
collapsed along the angular \(S^{d-1}\)-directions: for fixed \(t\ne0\) and
\(|y|=\sqrt{d+2}\,|t|\), one has
\[
        P_{a,c}(y,t)
        =
        \left(0,
        -\frac{2c(d+3)}d t^3
        \right),
\]
independently of the direction of \(y\). Since \(d\ge2\), this proves
non-injectivity.
\end{proof}

\begin{remark}
Theorem~\ref{thm-equivariant-cubic-classification} is a genuine
higher-dimensional nonexistence statement, but only in a symmetric class. It
shows that the most natural \(O(n-1)\)-equivariant cubic models cannot produce
a strict counterexample with \(J>0\) on the sphere. At the same time, when
\(n\) is even it gives nonnegative borderline Jacobians. Thus, the theorem
simultaneously strengthens the sharpness examples and explains why perturbing
the equivariant model is the first place where a strict higher-dimensional
counterexample would have to break symmetry.
\end{remark}

\begin{remark}
The classification shows that the dimension-three sign-change phenomenon does
not extend naively to higher dimensions. In even dimensions \(n\ge4\), there
are nontrivial homogeneous harmonic cubics with \(J\ge0\) on the sphere. What
remains open is the unrestricted strict problem: whether, for some \(n\ge4\),
there exists a homogeneous harmonic polynomial mapping \(P:\R^n\to\R^n\) of
degree \(m>1\) such that
\[
        J_P>0
        \qquad\text{on }S^{n-1}.
\]
\end{remark}

\begin{remark}
The planar case \(n=2\) shows that no dimension-free obstruction can hold
without qualification: the map \(z\mapsto z^m\) is homogeneous, harmonic,
and quasiregular, with a branch point at the origin.
\end{remark}

\medskip

\noindent\textbf{Acknowledgements.}
The authors thank Kai Rajala for pointing out the smooth quasiregular
local-invertibility theorem and for helpful comments on the formulation of the
main result.

\medskip

\noindent\textbf{Funding.} The first author gratefully acknowledges financial support from the Ministry of Education, Science and Innovation of Montenegro through the grants \emph{``Mathematical Analysis, Optimisation and Machine Learning''} and \emph{``Complex-analytic and geometric techniques for non-Euclidean machine learning: theory and applications.''} The second author was supported by NSF of China (No. 12271189), NSF of Guangdong Province (Grant No. 2024A1515010467, 2026A1515012333), STU Scientific Research Initiation Grant NTF25017T, HQU teaching reform project HQJGKT2411, and Fujian Alliance of Mathematics (Grant No. 2023SXLMMS07).

\medskip

\noindent\textbf{Data availability.} No datasets were generated or analyzed during the current study. Figure~\ref{fig:icosahedral-vertices} is an illustrative visualization of the icosahedral vertex configuration defined explicitly in Section~6.1. All information needed to reproduce the figure is contained in the manuscript.

\medskip

\noindent\textbf{Conflicts of interest.} The authors declare that they have no conflicts of interest regarding the publication of this paper.

\medskip

\noindent\textbf{Declaration of generative AI use.}
During the preparation of this work, the authors used OpenAI's ChatGPT to assist with language editing, exposition, and LaTeX formatting. After using this tool, the authors reviewed and edited the content as needed, verified the mathematical arguments and references, and take full responsibility for the content of the publication.


\begin{thebibliography}{99}

\bibitem{AstalaIwaniecMartin}
K. Astala, T. Iwaniec and G. Martin,
\emph{Elliptic Partial Differential Equations and Quasiconformal Mappings in
the Plane},
Princeton Mathematical Series, vol. 48, Princeton University Press,
Princeton, NJ, 2009.

\bibitem{AstalaManojlovic}
K. Astala and V. Manojlovic,
On Pavlovic's theorem in space,
\emph{Potential Anal.} \textbf{43} (2015), no. 3, 361--370.

\bibitem{BonkHeinonen}
M. Bonk and J. Heinonen,
Smooth quasiregular mappings with branching,
\emph{Publ. Math. Inst. Hautes \"Etudes Sci.} \textbf{100} (2004), 153--170.

\bibitem{BozinMateljevic}
V. Bozin and M. Mateljevic,
Bounds for Jacobian of harmonic injective mappings in \(n\)-dimensional
space,
\emph{Filomat} \textbf{29} (2015), no. 9, 2119--2124.

\bibitem{AxlerBourdonRamey}
S. Axler, P. Bourdon and W. Ramey,
\emph{Harmonic Function Theory},
2nd ed., Graduate Texts in Mathematics, vol. 137, Springer-Verlag,
New York, 2001.

\bibitem{Cheng1976}
S.-Y. Cheng,
Eigenfunctions and nodal sets,
\emph{Comment. Math. Helv.} \textbf{51} (1976), 43--55.

\bibitem{GleasonWolff}
S. Gleason and T. H. Wolff,
Lewy's harmonic gradient maps in higher dimensions,
\emph{Comm. Partial Differential Equations} \textbf{16} (1991), no. 12,
1925--1968.

\bibitem{Kalaj2015}
D. Kalaj,
Muckenhoupt weights and a Lindelof theorem for harmonic mappings,
\emph{Adv. Math.} \textbf{280} (2015), 301--321.

\bibitem{KalajSaksman2019}
D. Kalaj and E. Saksman,
Quasiconformal maps with controlled Laplacian,
\emph{J. Analyse Math.} \textbf{137} (2019), 251--268.

\bibitem{KaufmanTysonWu}
R. Kaufman, J. T. Tyson and J.-M. Wu,
Smooth quasiregular maps with branching in \(\R^n\),
\emph{Publ. Math. Inst. Hautes \`Etudes Sci.} \textbf{101} (2005), 209--241.

\bibitem{Lewy1936}
H. Lewy,
On the non-vanishing of the Jacobian in certain one-to-one mappings,
\emph{Bull. Amer. Math. Soc.} \textbf{42} (1936), no. 10, 689--692.

\bibitem{Lewy1968}
H. Lewy,
On the non-vanishing of the Jacobian of a homeomorphism by harmonic
gradients,
\emph{Ann. of Math. (2)} \textbf{88} (1968), no. 3, 518--529.

\bibitem{Martin2016}
G. J. Martin,
Harmonic degree one maps are diffeomorphisms: Lewy's theorem for curved
metrics,
\emph{Trans. Amer. Math. Soc.} \textbf{368} (2016), no. 1, 647--658.

\bibitem{MRV}
O. Martio, S. Rickman and J. V\"ais\"al\"a,
Topological and metric properties of quasiregular mappings,
\emph{Ann. Acad. Sci. Fenn. Ser. A I} \textbf{488} (1971), 1--31.

\bibitem{Pavlovic2002}
M. Pavlovic,
Boundary correspondence under harmonic quasiconformal homeomorphisms of the
unit disk,
\emph{Ann. Acad. Sci. Fenn. Math.} \textbf{27} (2002), 365--372.

\bibitem{Reshetnyak}
Yu. G. Reshetnyak,
\emph{Space Mappings with Bounded Distortion},
Translations of Mathematical Monographs, vol. 73, American Mathematical
Society, Providence, RI, 1989.

\bibitem{Rickman}
S. Rickman,
\emph{Quasiregular Mappings},
Ergebnisse der Mathematik und ihrer Grenzgebiete (3), vol. 26,
Springer-Verlag, Berlin, 1993.

\bibitem{Saunders}
D. J. Saunders,
\emph{The Geometry of Jet Bundles},
London Mathematical Society Lecture Note Series, vol. 142, Cambridge
University Press, Cambridge, 1989.

\bibitem{Vaisala}
J. Vaisala,
\emph{Lectures on \(n\)-Dimensional Quasiconformal Mappings},
Lecture Notes in Mathematics, vol. 229, Springer-Verlag, Berlin, 1971.

\bibitem{Wood1991}
J. C. Wood,
Lewy's theorem fails in higher dimensions,
\emph{Math. Scand.} \textbf{69} (1991), 166.

\bibitem{Zorich1967}
V. A. Zorich,
Homeomorphism of quasiconformal space mappings,
\emph{Soviet Math. Dokl.} \textbf{8} (1967), 1039--1042, translation from
\emph{Dokl. Akad. Nauk SSSR} \textbf{176} (1967), 31--34.

\end{thebibliography}
\end{document}